\documentclass[10pt]{amsart}
\usepackage{hyperref}
\usepackage{verbatim, amssymb,amsmath, amsfonts, amsthm, cases,hyperref,enumerate,tikzsymbols}

\newtheorem{theorem}{Theorem}[section]
\newtheorem{theoremcit}{Theorem}

\newtheorem{proposition}[theorem]{Proposition}
\newtheorem{lemma}[theorem]{Lemma}
\newtheorem{corollary}[theorem]{Corollary}
\newtheorem{remark}[theorem]{Remark}
\numberwithin{equation}{section}
\renewcommand{\(}{\left(}
\renewcommand{\)}{\right)}
 \newcommand{\intO}{\int_{ B}}

\def\R{{\rm I\mskip -3.5mu R}}
\def\N{{\rm I\mskip -3.5mu N}}

\def\e{{\varepsilon}}
\def\di12{\mathcal{D}^{1,2}(\R^n)}

\def\d{\delta}

\def\l{{\lambda}}

\def\a{{\alpha}}

\def\0l{_{0,\l}}
\def\1l{_{1,\l}}
\def\2l{_{2,\l}}
\def\3l{_{3,\l}}
\def\4l{_{4,\l}}

\def\de{\delta}

\def\Om{ B}

\providecommand{\norm}[1]{\lVert#1\rVert}

\newcommand{\rad}{\mathop{\mathrm{rad}}}
\newcommand{\loc}{\mathop{\mathrm{loc}}}
\newcommand{\AL}{\color{purple}}

\newcommand{\taglia}{\color{cyan}}
\definecolor{purple}{rgb}{0.5, 0.0, 0.5}
\definecolor{fran}{rgb}{0.1569,0.4471, 0.20}

\def\sideremark#1{\ifvmode\leavevmode\fi\vadjust{\vbox to0pt{\vss% the remark
 \hbox to 0pt{\hskip\hsize\hskip1em%                          will appear only
 \vbox{\hsize2.1cm\tiny\raggedright\pretolerance10000%          on the side
  \noindent #1\hfill}\hss}\vbox to15pt{\vfil}\vss}}}%

\usepackage{geometry}
\begin{document}
\title[Brezis Nirenberg problem]{
A complete scenario on nodal radial solutions to the Brezis Nirenberg problem in low dimensions \\
}

\author[]{A. L. Amadori, F. Gladiali, M. Grossi, A. Pistoia, G. Vaira}

\address{Anna Lisa Amadori, Dipartimento di Scienze e Tecnologie, Universit\`a di Napoli ``Parthenope", Centro Direzionale di Napoli, Isola C4, 80143 Napoli, Italy. \texttt{annalisa.amadori@uniparthenope.it}}
\address{Francesca Gladiali, Dipartimento di Chimica e Farmacia, Universit\`a  di Sassari  - Via Piandanna 4, 07100 Sassari, Italy. \texttt{fgladiali@uniss.it}}
\address{Massimo Grossi, Dipartimento di Matematica, Universit\`a degli Studi di Roma \emph{La Sapienza}, P.le Aldo Moro 2, 00185  Roma, Italy. \texttt{massimo.grossi@uniroma1.it}}
\address{Angela Pistoia, Dipartimento SBAI, Universit\`a degli Studi di Roma \emph{La Sapienza}, P.le Aldo Moro 2, 00185  Roma, Italy. \texttt{angela.pistoia@uniroma1.it}}
\address{Giusi Vaira, Dipartimento di Matematica, Università degli Studi di Bari Aldo Moro,
via Edoardo Orabona 4, 70125 Bari, Italy. \texttt{giusi.vaira@uniba.it}}



\thanks{Research partially supported by:  INDAM - GNAMPA}

\begin{abstract}
In this paper we consider nodal radial solutions of the problem
$$
\begin{cases}
-\Delta u=|u|^{2^*-2}u+\lambda u&\text{ in }B,\\
u=0&\text{ on }\partial B
\end{cases}
$$
where $2^*=\frac{2N}{N-2}$ with $3\le N\le6$ and $B$ is the unit ball of $\R^N$. We compute the asymptotics of the solution $u$ as  well as $||u||_\infty$, its first zero and other relevant quantities as $\l$ goes to a critical value $\bar\l$. Also the sign of $\l-\bar\l$ is established in all cases. 
 This completes an analogous analysis for $N\ge7$ given in \cite{EGPV}.
\end{abstract}

\maketitle

\section{Introduction}

In a celebrated paper, Brezis and Nirenberg \cite{BN83} proved the existence of $positive$ solutions of the following Sobolev critical problem,
\begin{equation}\label{1}
\begin{cases}
-\Delta u=|u|^{2^*-2}u+\lambda u&\text{ in }B,\\
u=0&\text{ on }\partial B
\end{cases}
\end{equation}
where $2^*=\frac{2N}{N-2}$ with $N\geq 3$, $B$ is the unit ball of $\R^N$ centered at the origin  and $\l$ is a real positive parameter. 
In \cite{BN83} it was proved that
  \begin{theoremcit}\label{teo:BN}
  The problem \eqref{1} admits a positive solution if and only if  $\l\in (0,\mu_1)$ in dimension $N\ge 4$, or respectively  $\l\in(\mu_1/4, \mu_1)$ in dimension $N=3$, where $\mu_1$ is the first eigenvalue of the Laplacian with Dirichlet boundary conditions.
  \end{theoremcit}
It is virtually impossible to give an exhaustive bibliography of all works inspired by this paper, so we will limit ourselves to mention only those related to the cases that interest us.\\
In this paper we consider {\em nodal radial} solutions to \eqref{1}. For any fixed $m\in \N^+$ denote by 
$u_\l^m$ a classical radial solution to \eqref{1} corresponding to the parameter $\l$, with $m$ nodal regions (i.e. $m-1$ internal zeros) and by $\mu_m$ the $m^{th}$ radial eigenvalue of the Laplacian in $B$ with Dirichlet boundary conditions.
 Here another new phenomenon appears, involving the dimension $N=7$. 
\begin{theoremcit}\label{teo-lambda}
	In dimension $N\ge 7$, problem \eqref{1} admits at least one radial solution $u_\l^m$ with $m$ nodal zones  for  $\l\in(0 ,\mu_m)$. 
	In dimension $N$ between $3$ and $6$, there exists $\l^* = \l^*(N,m)$ such that problem \eqref{1} does not admit any  radial solution with $m$ nodal zones for $\l\in (0, \l^*)$ and $\l>\mu_m$.
\end{theoremcit}
We refer to \cite{ABP90} for a proof of this facts; see also Figure 2 in the same paper. \\ 
Theorem \ref{teo-lambda} gives a complete answer to the problem of the existence of nodal radial solutions  to \eqref{1} in the ball. 

A first analysis of  the associated ordinary differential equation  (\cite{ABP90}, \cite{AP88}) brings into light another critical value, say it $\bar\l = \bar\l(N,m)$, with the properties:
 \begin{enumerate}[\it(i)]
 	\item   there is a sequence of radial solutions with $m$ nodal zones such that   $\|u^m_\l\|_{\infty}\to \infty$ as $\l\to \bar\l$.
 		\item[\it(ii)] if $\|u^m_{\l}\|_{\infty}\to \infty$, then (\cite{ABP90,BP89,Han91})
 \begin{itemize}		
 \item 	$\l\to\bar\l=\left(\frac{2m-1}2 \pi \right)^2$ in dimension $N=3$,
 \item $\l\to\bar\l=\mu_{m-1}$ in dimension $N=4,5$,
\item $\l\to\bar\l\in(0,\mu_{m-1})$ in dimension $N=6$,
\item  $\l\to0$ in dimension $N\ge7$,
\item  $\frac 1{\| u_\l^m\|_\infty}u_\l^m\left( \frac x{\|u_\l^m\|_\infty^{\frac 2{N-2}}}\right)\to U(x)$ in $C^1_{\loc} (\R^N)$, 
 \end{itemize}	
	\end{enumerate}
 where $U$ is the unique  solution to the Sobolev critical problem \begin{equation}\label{eq:probl-crit}
 	\begin{cases} -\Delta U=U^{2^*} & \text{ in }\R^N, \\
 		U>0 & \text{ in }\R^N, \\
 		 U(0)=1. & \end{cases}\end{equation}
 	 We shall refer at this number  $\bar\l$ as a {\it concentration value} and we recall that in the last decades there has been a lot of work about the properties of the solutions $u^m_\l$ near $\bar\l$.\\
 If $N\ge7$ the radial solutions have a multiple blow up at the origin. The asymptotic profile in this case has been studies first in \cite{Iac15} (concerning solution with two nodal zones) and then completed in \cite{EGPV} in the general case.\\
In order to state this result, let us introduce the following notations,
\begin{equation}
M_{i,\l}=\sup\limits_{r\in (r_{i-1,\l},r_{i,\l})}|u_\l^m(r)|
\end{equation}
and
\begin{equation}
r_{i,\l}\in(0,1],\ i=1,..,m-1,\hbox{ be the nodal radii of $u_\l^m$, i.e } u_\l^m(r_{i,\l})=0.
\end{equation}
We have,
\begin{theoremcit}\label{thm:EGPV}(see \cite{EGPV}). Let $u_\l^m$ be a radial solution to \eqref{1} with exactly $m$ nodal zones in dimension $ N\ge 7$.
	 When $\l\to 0$ we have
	\[\begin{array}{c} M_{i,\l} \to \infty, \quad r_{i,\l}\to 0,  
	\\[4pt]
 M_{i,\l}^\frac 2{N-2} r_{i,\l}\to \infty, \quad   
 \\[4pt]
\frac 1{M_{i,\l} } u_\l^m\left(\frac x{M_{i,\l}^\frac 2{N-2}}\right)   \to (-1)^{i-1} U(x) \ \text{ in } C^1_{\loc}(\R^N\setminus\{0\})
\end{array} \]
for $i=1,\dots m$, where $U$ is the unique  solution to \eqref{eq:probl-crit}. Furthermore the asymptotic rates of $M_{i,\l}$ and $ r_{i,\l}$ with respect to $\lambda$ are computed.
\end{theoremcit}
From this theorem we get that the solution $u_\l^m$ exhibit  a {\em tower of bubbles}, i.e. $m-1$ nodal lines collapse to the origin and in any nodal region the solution blows-up like the scaled bubble $U$.\\
 In this paper we are interested in the remaining open cases $N=3,4,5,6$ and our aim is to prove the corresponding estimates
 as in \cite{EGPV}.  In particular we  see that the  tower of bubbles behaviour in Theorem \ref{thm:EGPV} is characteristic of the high dimensions, $N\geq 7$. When $N<7$ indeed the solution blows up in the origin as $\l\to \bar \l$ like the scaled bubble $U$, but stays bounded in the interval $(r_{1,\l},1)$. Hence it is needed to analyze only the first zero $r_{1,\l}$, that we denote hereafter by $r_\l$, and the first critical critical value in $(r_{1,\l},1)$, say $M_\l$, since all the others converge to the subsequent zeroes and critical values of the limit problem.

In particular we are interested in the following questions:
\vskip0.2cm
{\em Question 1. What about the asymptotic behavior of the solution $u^m_\l$ as $\l\to\bar\l$?}
\vskip0.2cm
The answer to this question concerns the computation of the rate of $\|u^m_{\l}\|_{\infty}$, $M_{i,\l}$ and $r_{i,\l}$.   As we said before it is known that $\|u^m_{\l}\|_{\infty}\to+\infty$ but the rate (in term of $\l-\bar\l$) is not known as well as for the other quantities.
\vskip0.2cm
{\em Question 2. What about the sign of the difference $\l-\bar \l$?}
\vskip0.2cm
The motivation of the second question comes from the (quite natural) conjecture that $\bar\l$  coincides with the threshold value for existence, that is $\l^*$. Seemingly it is not the case, or better also this issue seems to depend upon the dimension. Indeed it was proved in \cite{GG} that $\l-\bar \l<0$ for $N=5$ while $\l-\bar \l>0$ when $N=4$ and in this last case $\bar \l=\l^*$, see \cite{Arioli-Gazzola-Grunau}. Further, when $N\ge 7$, $\l>\bar\l=0$ trivially. The problem is still open in the remaining dimensions, namely $N=3$ and $N=6$ and we will give an answer.
This point will be discussed in detail later.\\
We want to stress that, while previously all the results on the  asymptotic behavior of $u_\l^m$  used the Emden-Fowler transformation, here a blow-up analysis jointly with some integral identities
is used. This different approach allows us to give an alternative proof of some known facts. We also point out that it has been  successfully used in studying the properties of solutions to the classical Yamabe problem in a series of papers, see for example  \cite{D,DHR,rv} and the references therein.
 \\
Now we consider separately the different dimensions, since each of them has different features.
\vskip0.2cm
{\bf The case N=3}
\vskip0.2cm
In this case, the only known facts are that $\l\to\bar\l=\left(\frac{2m-1}2 \pi \right)^2$ (see introduction in \cite{ABP90}) and, if $m=2$, $ r_{\l}\to\bar r=\frac13$, $u^m_\l\to0$ in $(0,1)$ (see \cite{IP}). In next result we compute the asymptotics of $u^m_\l$ and we show that $u^m_\l$, up to a suitable normalization, converges to the radial eigenfunction $z$ of $-\Delta$ in the annulus $(\bar r,1)$. This gives that the nodal radii $r_{i,\l}$ converge to the zeros of $z$ as well as for the critical values. Lastly we show that $\l\to\bar\l$ from $above$.
\begin{theorem}\label{teo:N=3} 
	Let  $u_\l^m$ be any radial nodal solution to \eqref{1} with $m>1$ nodal zones in dimension $N=3$ and  let $r_\l$ be its first zero in $(0,1)$. 
	The occurrence $\|u_\l^m\|_{L^{\infty}(B)}=u_\l^m(0) \to \infty$ 
		can happen only if $\l\to \bar\l=\left(\frac{2m-1}2 \pi\right)^2$ from above. Moreover, 
	as $\l\to \bar \l$ we have 
	\begin{align}
	\label{norma-B-N=3} 
	\|u_\l^m\|_{L^{\infty}(B)} & = \sqrt[4]{3}\sqrt{ \frac{(2m-1)^3\pi^3}{8m-6}}\,(1+o(1)) \frac{1}{\sqrt{\l- \bar \l}}\, ,  \\
	\label{r-star}
	r_\l  &  =\frac 1{2m-1}+ \frac{8(m-1)}{\pi^2 (2m-1)^3}\,(1+o(1)) (\l-\bar \lambda),  \\ 
	 \dfrac{u_{\l}^m( x) }{\sqrt {\l-\bar \l}} & \to 4\sqrt[4]{3}\sqrt{\frac{2(4m-3)}{2m-1}} 
\,  V\big((2m-1) x,0\big) \quad   \text{ in }C^1_{\loc} \left(\bar{B}_{\frac 1{2m-1}}\setminus \{0\} \right),\label{inseire}
	\end{align}
	where $V(x,0)$ is the Green function of the operator $-\Delta-\frac {\pi^2}4I $ with Dirichlet boundary condition on the unit ball and	
	\begin{align}
		\label{u-A-N=3}
\frac{u_\l^m(x)}{\sqrt{\l-\bar\l}} 
& \to -4\sqrt[4]{3}\sqrt{\frac{2(4m-3)}{2m-1}}V'\!(1,0)\frac2{(2m-1)\pi |x|}\cos\left(\frac {2m-1}2 \pi |x|\right)
\   \text{ in } C^1_{\loc}\left(\bar{B}\setminus \bar{B}_{\frac 1{2m-1}} \right).
	\end{align}	
		\end{theorem}
\vskip0.2cm
{\bf The case N=4,5}
\vskip0.2cm
In this case in \cite{ABP90} it was proved that $\l\to\bar\l=\mu_{m-1}$  and $u_\l^m$ converges to the $(m-1)$-$th$ radial eigenfunction of the Laplacian in $B$, away from the origin (see the proof of Theorem B, part (a) in \cite{ABP90} for this last claim). Here we compute the asymptotics of $||u_\l^m||_\infty$, $r_{\l}$, $M_{\l}$ and we show that $\l\to \mu_{m-1}$ from above in dimension $N=4$ and from below in dimension $N=5$. This last result was already proved by  Gazzola and Grunau in \cite{GG} using the Emden-Fowler transformation and so we give an alternative proof based only on the blow-up analysis. Our result is the following:
\begin{theorem}\label{teo:N=4,5} 
	Let  $u_\l^m$ be any radial nodal solution to \eqref{1} with $m>1$ nodal zones in dimension $N=4$ or $N=5$ and let $r_\l$ be its first zero in $(0,1)$.
	The occurrence $\|u_\l^m\|_{L^{\infty}(B)} =u_\l^m(0)\to \infty$ is equivalent to $r_\l \to 0$ and can happen only if $\l\to \mu_{m-1}$ from above in dimension $N=4$ and from below in dimension $N=5$.
	Moreover, as $\l\to \mu_{m-1}$ 
	\begin{align}
\|u_\l^m\|_{L^{\infty}(B)} & =
\begin{cases}
\displaystyle { \frac { 16} {\int_0^1 r^3 |\psi_{m\!-\!1}|^2 dr} (1+o(1)) \, (\l-\mu_{m\!-\!1})^{-1}} , & \text{ if } N=4 ,\\
\displaystyle   \left(\frac {5\pi\mu_{m\!-\!1}}{8}\right)^3 
\left(\frac{\int_0^1 r^4 |\psi_{m\!-\!1}|^{\frac{10}3}dr} {\int_0^1r^4|\psi_{m\!-\!1}|^2 dr}\right)^{\frac{9}{4}}
\!\! (1+o(1)) \,  (\mu_{m\!-\!1}-\l)^{-\frac{9}{4}} , & \text{ if } N=5 ,
	\label{r-N=5}
\end{cases}
\\
\label{r-N=4,5}
r_\l & =
 \begin{cases}
\displaystyle \sqrt{\frac{2}{\mu_{m\!-\!1}}} (1+o(1)) \left|\log (\l-\mu_{m\!-\!1}) \right|^{-\frac{1}{2}} , & \text{ if } N=4, \\
\displaystyle \frac{8\sqrt{3}}{\pi\mu_{m\!-\!1}\sqrt{5}} \left(\frac{\int_0^1r^4|\psi_{m\!-\!1}|^2 dr}{\int_0^1 r^4 |\psi_{m\!-\!1}|^{\frac{10}3}dr}\right)^{\frac{1}{2}} \!\! (1+o(1)) \, (\mu_{m\!-\!1}-\l)^{\frac{1}2}  , & \text{ if } N=5.
\end{cases}
\end{align}
Here $\psi_h$ denotes the $h$-$th$ radial eigenfunction of the Laplacian in $B$, normalized so that $\psi_h(0)=-1$.
Furthermore, denoting by $A_\l$ the annulus of radii $r_\l$ and $1$, we have
	\begin{align}
\label{norma-A-N=4}
\norm{u_\l^m}_{L^{\infty}(A_\l)}  & = 
\begin{cases}  
\displaystyle \frac{\mu_{m\!-\!1}}{4} \int_0^1 r^3|\psi_{m\!-\!1}|^2 dr \, (1+o(1)) (\l-\mu_{m\!-\!1})|\log(\l-\mu_{m\!-\!1})|  , & \text{ if } N=4 , \\
 \displaystyle \left(\frac{\int_0^1 r^4|\psi_{m\!-\!1}|^2 dr}{\int_0^1 r^4|\psi_{m\!-\!1}|^{\frac{10}3} dr}\right)^{\frac{3}{4}} (1+o(1)) (\mu_{m\!-\!1}-\l)^{\frac{3}{4}} , & \text{ if } N=5 ,
\end{cases}
\\
\label{u-A-N=4}
\dfrac{u_{\l}^m( x) }{\norm{u_\l^m}_{L^{\infty}(A_\l)}} & \longrightarrow \psi_{m-1}(x) \qquad  \text{ in } C^1_{\loc}\left(\bar{B}\setminus \{0\} \right).
\end{align}
 \end{theorem}

{\bf The case N=6}
\vskip0.2cm
 This is the most delicate case. A first question concerns the characterization of the concentration value $\bar\l$. We will see that $\bar\l=\bar\l(m)$ is characterized as the unique value at which there exists a radial solution $u^{m-1}$  
of the problem
\begin{equation}\label{i10}
\begin{cases}  
-\Delta u^{m-1}=|u^{m-1}|u^{m-1}+\bar \l u^{m-1}&in\ B\\
u^{m-1}(0)=-\frac{\bar\l}2\\
u^{m-1}\hbox{ has $m-1$ nodal zones }\\
u^{m-1}=0&on\ \partial B.
\end{cases} 
\end{equation}
 We will show in Proposition \ref{un} that \eqref{i10} admits a radial solution only for a unique value of $\bar \l$, providing the characterization of $\bar \l$. Moreover since this solution is unique, problem \eqref{i10} characterizes $u^{m-1}$ as well.
One can also see \cite[Section 5]{ABP90} or \cite[Theorem 4]{IP}, where the definition of $\bar \l$ is related to the Emden-Fowler transformation. Next we will prove that any radial nodal solution $u_\l^m$ to \eqref{1} converges to the solution $u^{m-1}$ of \eqref{i10} in $C^1(B\setminus0)$, generalizing \cite[Theorem 3]{IP}  to the case of $m>1$, see also \cite[Section 5]{ABP90}.  Observe that, in this case $M_\l\to \frac{\bar \l}2$. Finally we will compute the asymptotics of the relevant quantities $||u_\l^m||_\infty$ and  $r_{\l}$ as in the previous cases, and  we characterize the sign of $\l{\taglia \to} {\AL -}\bar\l$ (in particular it {\AL is} positive if $m=2$).
%This restriction arises because we use that the $positive$ solution $u^1$ has Morse index $1$ and this is not true for solutions with more zeros.\\
 In order to get our results we proceed differently from the cases $N=3,4,5$. 
In fact, despite the blow-up procedure shows no differences, the integral identities of the previous cases do not allow us to obtain the desired result. So we argue differently: we first construct a solution to \eqref{1} using the Ljapunov-Schmidt procedure, next we deduce the asymptotics for this solution and finally we prove the uniqueness of the solution $u_\l^m$ in the class of the blowing-up solutions. A crucial role in our result is played by the solution $v_0$ of the problem\\
\begin{equation}\label{v0i}
\begin{cases}
-\Delta v_0-(2 |u^{m-1}|+\bar \lambda )v_0=u^{m-1}\quad\mbox{in}\ B\\
v_0=0\ \hbox{on}\ \partial B
 \end{cases}
 \end{equation}
which exists and it is unique if the solution $u^{m-1}$ to \eqref{i10} is nondegenerate. Our result is the following:
\begin{theorem}\label{teo:N=6}
Let  $u_\l^m$ be any radial nodal solution to \eqref{1} with $m>1$ nodal zones in dimension $N=6$ and let $r_\l$ be its first zero in $(0,1)$.
	The occurrence $\|u_\l^m\|_{L^{\infty}(B)} =u_\l^m(0)\to \infty$ is equivalent to $r_\l \to 0$ and can happen only if $\l\to \bar \l$ where $\bar \l$ is the unique value such that problem \eqref{i10} has a radial solution.
	Moreover, denoting by $A_\l$ the annulus of radii $r_\l$ and $1$, as $\l\to \bar\l$ we have
	\begin{align}
	\label{norma-A-N=6}
	\norm{u_\l^m}_{L^{\infty}(A_\l)}   & \to   \frac{\bar\l}2 
\\
	\label{u-A-N=6}
	u_\l ^m(x)   & \to  u^{m-1} (x) \quad  \text{ in } C^1_{\loc}(\bar B\setminus\{0\}), 
	\end{align}	
	where $u^{m-1}$ stands for the unique radial solution to \eqref{i10}.
	\\
Furthermore, either if $m=2$ or if $u^{m-1}$ is nondegenerate, as $\l\to \bar\l$  
we have that 
	\begin{align}
	\label{norma-B-N=6} 
	\|u_\l^m\|_{L^{\infty}(B)} & =\frac{121\, \bar \l^3}{8\big(1+2v_0(0)\big)^2}(1+o(1)) (\l-\bar\l)^{-2} 
 \\
	\label{r-N=6}
	r_\l & = 4\sqrt{\frac{6}{11}} \frac{\big|1+2v_0(0)\big|^\frac12}{\bar \l
	} (1+o(1)) |\l-\bar\l|^\frac12&	 
	\end{align}
	where $v_0$ is the unique solution to \eqref{v0}  and the following expansion of the solution $u_\l^m$ holds, 
	\begin{equation}\label{expansion}
		u_\l^m(r)=u^{m-1}(r)+(\l-\bar \l)v_0(r)+PU_{(\l-\bar \l) d}(r)+\phi_\l (r)\end{equation}
	where $PU_\delta$ is the projection of the standard bubble $U_\delta(x)=\delta^{-2}U\left(\frac x \de\right)$ onto $H^1_0(B)$ (see \eqref{Udxi} and \eqref{PU}), $d$ is a positive number (see \eqref{scelta-d}) and $\phi_\l\in H^1_0(B)$ is such that $\|\phi_\l\|_{H^1_0(B)}=\mathcal O\((\l-\bar\l)^2|\ln|\l-\bar\l||^{\frac23}\)$.
	Finally if $m=2$ then $\l\to \bar \l$ from above, while when $m>2$ we have that $\l-\bar \l>0$ if $1+2v_0(0)>0$ while $\l-\bar \l<0$ when $1+2v_0(0)<0$.
\end{theorem}
\begin{remark}  
When $m=2$ the nondegeneracy of the positive solution to  \eqref{i10} was proved in \cite{Sri93}. So the previous theorem gives a complete scenario of  the asymptotics of the solution $u_\l^m$ as $m=2$. Observe that $1+2v_0(0)\neq 0$ for every $m\geq2$ and this will be proved in Proposition \ref{prop-4}.
\end{remark}

\begin{remark}  
 We point out the careful construction of the ansatz \eqref{expansion}  which has to be  refined up to a second order   and the  delicate estimate of the reduced energy  \eqref{cruc1} given in Proposition \ref{cruc0}   whose leading term  arises from  the interaction between the bubble and the second order term in the ansatz.\\
\end{remark}

The paper is organized as follows: in Section \ref{s2} we recall some known results about positive solutions of \eqref{1} and prove some general properties of nodal solutions; in Section \ref{s3} we prove Theorems \ref{teo:N=3} and \ref{teo:N=4,5}. Finally in Section \ref{s4} we prove Theorem \ref{teo:N=6} .

 \section{Known facts and preliminary remarks}\label{s2}
 
 In this section we recall some known facts about radial solutions to the Brezis-Nirenberg problem and we fix the notations that will be used in the paper.  
From now on we will delete the index $m$ and only write 
 $$u_\l^m\equiv u_\l.$$
 We start considering the case of positive solutions that has been extensively studied in the 80’s, mainly by Brezis, Nirenberg, Peletier, Atkison, see \cite{ BN83,ABP90, AP88}.
 Any positive solution is radial and radially decreasing (by the symmetry result in  \cite{GNN79}), and  is unique (see \cite{Sri93}), therefore it is a least energy solution and satisfies 
 \begin{equation}\label{S-lambda}
 	\displaystyle	S_\l:= \dfrac{\int_B |\nabla u_\l|^2 dx - \l \int_B |u_\l|^2 dx}{\left(\int_B |u_\l|^{\frac{2N}{N-2}} dx \right)^{\frac{N-2}{N}}}
 	= \inf\limits_{\stackrel{\phi\in H^1_{0,\rad}(B)}{\phi\neq 0}}  \dfrac{\int_B |\nabla \phi|^2 dx - \l \int_B \phi^2 dx}{\left(\int_B \phi^{\frac{2N}{N-2}} dx \right)^{\frac{N-2}{N}}} 
 \end{equation}
 for every $\l\in (0,\mu_1)$ in dimension $N\ge 4$, or respectively  $\l\in(\mu_1/4, \mu_1)$ in dimension $N=3$, from Theorem \ref{teo:BN}.
 By \cite[Lemma 1.1 and 1.3]{BN83}  we get
 \begin{lemma}\label{lemma-bn}
 	For every $\l>0$ when $N\geq 4$ or for every $\l>\mu_1/4={\pi^2}/4$ when $N=3$
 	\begin{equation}
 		S_\l<S_N
 	\end{equation}
 	where $S_N$ is the best constant for the Sobolev embedding $H^1_{0}(B)\subset L^{2^*}(B)$. 
 \end{lemma}
 Next we describe the blow up rate of the positive solution. In dimension $N=3$, the profile is linked to the Green function of the operator $-\Delta -\frac{\pi^2}4I$ on the unit ball, namely
 \begin{equation}\label{green-pi}\begin{cases}
 		-\Delta V(x,0) -\frac{\pi^2}4 V(x,0) = \delta_0 & \text{ in } B \\
 		V (x,0)= 0 & \text{ on } \partial B 
 \end{cases}\end{equation}
 where $\d_0$ is the Dirac mass centered at the origin. 
 By \cite[Theorem 3]{BP89} it follows that
 \begin{theorem}\label{thm:BP}
 	Let  $u_\l$ be the positive solution to \eqref{1}  in dimension $N=3$.
 	Then $ \norm{u_\l}_{\infty}\to\infty$ as $\l\to\pi^2/4$ and precisely 
 	\begin{equation}\label{norma-pos-N=3}
 		u_{\l}(0)= \norm{u_\l}_{\infty}= \sqrt{\frac{\pi^3}2}\sqrt[4]3 \,   \big(1+o(1)\big) \frac1{\sqrt{\l-\pi^2/4}} \ \text{ as $\l\to\pi^2/4$}
 	\end{equation}
 	and 
 	\begin{equation}\label{u-pos-N=3}
 		\dfrac{u_{\l}(x)}{\sqrt{\l-\pi^2/4} } \longrightarrow  4\sqrt 2\sqrt[4]3
 		\,  V(x,0) \quad \text{ in } C^1_{\loc}(\overline{B}\setminus\{0\}) \ \text{ as $\l\to\pi^2/4$}.
 	\end{equation} 
 \end{theorem}

 Han \cite{Han91} dealt with the higher dimensional case, and 
 proved that the limit profile is driven by the Green function  of the Laplacian on the unit ball, i.e. 
 \begin{equation}\begin{cases}
 		-\Delta G(x,0) = \delta_0 & \text{ in } B \\
 		G(x,0) = 0 & \text{ on } \partial B
 \end{cases}\end{equation}
 solving also Conjecture 2  in \cite{BP89}.
 Han's result is the following,
 \begin{theorem}\label{thm:Han}
 	Let  $u_\l$ be the positive solution to \eqref{1}  in dimension $N\ge 4$.
 	Then $ \norm{u_\l}_{\infty}\to\infty$ as $\l\to0 $ and precisely
 	\begin{equation}\label{norma-pos-N=4,5,6} 
 		\begin{cases} 
 			\norm{u_\l}_{\infty}= C_N (1+o(1)) \, \l^{-\frac{N-2}{2(N-4)}}  \quad & \text{ if } N\ge 5, \\
 			\log  \norm{u_\l}_{\infty}= 2(1+o(1)) \, \l^{-1}& \text{ if } N=4
 		\end{cases}
 	\end{equation}
 	where 
 	\[C_N=
 	[N(N-2)]^{\frac{N-2}4}\left[ \frac{(N-2)^2}{2a_N}\right]^{\frac{N-2}{2(N-4)}} \qquad a_N=\int_0^\infty \frac {r^{N-1}}{(1+r^2)^{N-2}} dr .
 	\]
 	Moreover, letting $\sigma_N$ be the measure of the sphere $S^{N-1}\subset \R^N$
 	\begin{equation}\label{u-pos-N>3}
 		\|u_\l\|_{\infty}  u_{\l}(x)\to  [N(N-2)]^{\frac{N-2}2}
 		(N-2)\sigma_N\, G (x,0)\quad \text{ in } C^1_{\loc}(\overline{B}\setminus\{0\}) \ \text{ as $\l\to 0$. }
 	\end{equation}
 \end{theorem}	
 Now we consider $nodal$ solutions and, exploiting the results about positive solutions just recalled, we give a rough description of 
 the concentrating phenomenon.
 
 As a preliminary, we recall some general qualitative properties of the radial solutions to \eqref{1}.
 As usual we write $u_\l(r)$ for $u_\l(x)= u_\l(|x|)$ (meaning $r=|x|$), and 
 $0<r_{1,\l}< r_{2,\l}<\dots r_{m,\l}=1$ for the zeros of $u_\l$.
 Writing \eqref{1} in radial coordinates gives an ordinary differential equation with mixed boundary data:
 \begin{equation}\label{radial}
 	\begin{cases}
 		-(r^{N-1}u_\l')'=r^{N-1}\left(|u_\l|^{2^* -2}+\l u_\l\right) ,   &  \text{in }(0,1) ,\\
 		u_{\l}'(0)=0, \quad  u_{\l}(1)=0 . &
 	\end{cases}
 \end{equation}	
 Starting from this, it is easy to see that in each nodal interval the function $u_{\l}$ is alternately strictly positive or strictly negative and has  only one critical point, which is respectively a maximum or a minimum. 
 Moreover the extremal values are ordered
 \begin{equation}\label{massimi-ordinati} 
 	|u_\l(0)| = \max\limits_{[0,r_{1,\l})} |u_\l| > \max\limits_{(r_{1,\l}, r_{2,\l})}|u_\l| > \dots > \max\limits_{(r_{m\!-\!1,\l}, r_{2,\l})}|u_\l| ,
 \end{equation}
 see \cite[Lemma 1]{ABP90}.  
 Henceforth we shall always take that $u_\l$ is positive in the first nodal zone $[0, r_{1,\l})$
 	\begin{equation}\label{b1}
 		u_\l (0)>0, 
 \end{equation}
 and use the notations
 \begin{enumerate}[{\bfseries -}] 
 	\item $r_\l := r_{1,\l}$ for the first zero of $u_\l$,
 	\item  $s_\l$ for point where $u_\l$ attains its global minimum, which actually is the first minimum of $u_\l$, 
 	\item $A_\l:= \{ x \, : \, r_\l<|x|<1 \} = B\setminus \bar B_{r_{\l}}$.
 \end{enumerate} 
 By the previous consideration we have
 \begin{align}\label{remark-vecchio}
 	\|u_\l\|_{L^{\infty}(B)} = u_\l (0), \qquad \|u_\l\|_{L^{\infty}(A_\l)} =- u_\l (s_\l) .
 \end{align}

 Let us remark that a simple scaling argument and the non-existence result recalled in Theorem \ref{teo:BN}, implies
 \begin{lemma}\label{oss-1}
 	Let $u_\l$ be any radial nodal  solution to \eqref{1} and $r_{\l}$ its first zero in $(0,1)$. Then
 	\begin{align}
 		0< \l r_\l^2 < \mu_1 \quad & \text{ if } N=4, 5, 6 , \\
 		\frac{\pi^2}{4} < \l r_\l^2 < \pi^2 \quad & \text{ if } N= 3. \label{f2a}
 	\end{align}
 \end{lemma}
 \begin{proof}
 	Let \begin{equation}\label{e-l}
 		v_\e(x):=r_{\l}^{\frac{N-2}2}u_{\l}(r_{\l}x),  \qquad \e = \e(\l)=\l r_\l^2.
 	\end{equation}
 	A simple computation shows that $v_\e$ solves 
 	\begin{equation}\label{BN-pos}
 		\begin{cases}
 			-\Delta v = v^{\frac{N+2}{N-2}} + \e v & \text{ in } B , \\
 			v>0 & \text{ in } B , \\
 			v= 0 & \text{ on } \partial B.
 		\end{cases}
 	\end{equation}
 	Then Theorem \ref{teo:BN} gives the claim recalling that $\mu_1=\pi^2$ when $N=3$.
 \end{proof}
 
 Starting from the knowledge of the positive solution, one can study  the behaviour of the first node $r_\l$ for $\l$ close to the concentration value  $\bar\l$; we see that the first nodal zone collapses in dimension $N\ge 4$, while it does not vanish in dimension $N=3$.
 
 \begin{lemma}\label{lem-concentrazione}
 	Let  $u_\l$ be any radial nodal solution to \eqref{1} for $N=3,4,5,6$,   $r_{\l}$ its first zero in $(0,1)$, and 
 	$\bar\l$  such that  $\|u_{\l}\|_{\infty}=u_\l(0) \to \infty $ when $\l \to \bar \l$, then
 	\begin{equation}\label{r-0}
 		\lim\limits_{\l\to\bar \l}  r_{\l} = \begin{cases} 0 \qquad & \text{ if } N\ge 4 , \\
 			\sqrt{\dfrac{\mu_1}{4\bar\l}}= \dfrac{\pi}{2\sqrt{\bar\l}} > 0 \quad & \text{ if } N=3 . \end{cases}
 	\end{equation}
 \end{lemma} 
 \begin{proof}
 	Recall that, by Theorem \ref{teo-lambda}, $\bar \l>0$ when $3\leq N\leq 6$.
 	As in the previous Lemma we use the function $v_\e(x):=r_{\l}^{\frac{N-2}2}u_{\l}(r_{\l}x)$ that satisfies \eqref{BN-pos} for $\e=\l r_\l^2$ and we write $\bar \e = \lim\limits_{\l\to\bar\l} \e$. We claim that $\|v_\e\|_{\infty}\to \infty$ as $\e\to \bar \e$ if and only if either $\bar\e =\mu_1/4$ when $N=3$ or $\bar \e= 0$ when $N\ge 4$. Using this claim we can conclude the proof. Indeed in the case when $N=3$, \eqref{f2a} gives that $ \lim\limits_{\l\to\bar\l} r_\l>0$ so that
 	$\|v_\e\|_{\infty}=v_\e (0)=r_{\l}^{\frac12}u_{\l}(0)\to+\infty$ as $\e\to \bar \e$. The previous claim then implies that $\bar \e=\mu_1/4= \pi^2/4$ and gives \eqref{r-0}. \\
 	When $N\geq 4$ instead  we assume by contradiction that \eqref{r-0} does not hold. Then again $\lim\limits_{\l\to\bar\l} r_\l>0$ and
 	$\|v_\e\|_{\infty}=r_{\l}^{\frac{N-2}2}u_{\l}(0)\to+\infty$ as $\e\to \bar \e$. 
 	Then the claim gives that $\bar \e=0$ which implies 
 	$r_\l\to 0$ since we know that $\bar \l>0$. This contradiction concludes the proof of \eqref{r-0}.\\
 	Finally we prove the claim. Of course it is a known result for positive solutions, but we report a proof for completeness. 
 	One implication is already stated in Theorem \ref{thm:BP} for $N=3$ and in Theorem \ref{thm:Han} for $4\le N\le 6$. Next we assume that $\bar \e>\mu_1/4$ when $N=3$ and $\bar \e>0$ when $N\geq 4$. Let us check first that $S_\e \to S_{{\bar\e}}$ as $\e\to {\bar\e}$ where as said in \eqref{S-lambda}
 	\[
 	S_\e =\inf\limits_{\stackrel{\phi\in H^1_{0,\rad}(B)}{\phi\neq 0}}  \dfrac{\int_B |\nabla \phi|^2 dx - \e \int_B \phi^2 dx}{\left(\int_B \phi^{\frac{2N}{N-2}} dx \right)^{\frac{N-2}{N}}}
 	= \frac{\int_B |\nabla v_{\e}|^2 dx - \e \int_B |v_{\e}|^2 dx}{\left(\int_B |v_{\e}|^{\frac{2N}{N-2}} dx \right)^{\frac{N-2}{N}}   }.
 	\]    
 	By definition 
 	\[	\begin{split}
 		S_\e& \le \frac{\int_B |\nabla v_{\bar\e}|^2 dx - \e \int_B |v_{\bar\e}|^2 dx}{\left(\int_B |v_{\bar\e}|^{\frac{2N}{N-2}} dx \right)^{\frac{N-2}{N}}}
 		\\
 		&=S_{{\bar\e}}+\frac{(\e-{\bar\e})\int_{B}v_{\bar\e}^2}{\left(\int_{B}v_{\bar\e}^\frac{2N}{N-2}\right)^\frac{N-2}N}  \underset{\text{Holder}}{\le} S_{{\bar\e}}+|\e-{\bar\e}| |B|^{\frac{2}{N} } ,
 	\end{split}   
 	\]    
 	where $|B|$ stands for the measure of the ball $B$. Similarly one sees that $S_{\bar\e}\le S_{\e}+ |\e-{\bar\e}| |B|^{\frac{2}{N} }$, so that 
 	$ | S_\e - S_{{\bar\e}} | \le |\e-{{\bar\e}}|  |B|^{\frac{2}{N}}$ and the claim is proved. In particular, by Lemma \ref{lemma-bn}, $\lim\limits_{\l\to\bar\l} S_\e =S_{\bar \e}< S_N$ and this last fact implies the compactness of $u_\e$ which ends the proof.
 \end{proof}
 
 The asymptotics of the positive solutions recalled in Theorems \ref{thm:BP} and \ref{thm:Han}, together with a scaling argument, allows us to obtain in a simple way an estimate of the $L^{\infty}$-norm of $u_\l$ in term of its first zero $r_{\l}$ which will be very useful in the sequel.
 First we deal with the case of dimension $N=3$
 
 \begin{lemma}\label{lem-preliminare-N=3} 
 	Let  $u_\l$ be any radial  nodal solution to \eqref{1} with $m\geq 2$ nodal zones in dimension $N=3$,  $r_\l$ its first zero in $(0,1)$, and $\bar \l>0$ such that $ \norm{u_\l}_{\infty}=u_\l(0)\to\infty$ when $\l\to \bar \l$.  Then
 	\begin{equation}\label{norma-N=3} 
 		\norm{u_\l}_{\infty} =\frac{\pi \sqrt[4]{3\bar\l}  \big(1+o(1)\big) } {\sqrt{\l r_\l^2-\pi^2/4}}\ \ \ \hbox{ as $\l\to \bar \l$}.
 	\end{equation}
 	Moreover, denoting by $V(x,0)$  the solution of problem \eqref{green-pi} and by $\bar r = \frac {\pi}{2\sqrt{\bar\l}}= \lim\limits_{\l\to\bar\l} r_\l$,	we have that as $\l\to \bar \l$
 	\begin{equation}\label{u-prima zona-N=3} 
 		\dfrac{u_{\l}(  x)}{\sqrt{ \l r_{\l}^2-\pi^2/4 }  } \longrightarrow  \frac{ 8\sqrt[4]{3\bar \l}}{\sqrt{\pi}} \,  V\left(\frac{x}{\bar r},0\right) \quad \text{ in } C^1_{\loc}( \overline{B_{\bar r}}\setminus\{0\}) .
 	\end{equation} 
 	Finally
 	\begin{equation}\label{stima-der-prima-N=3} 
 		u_{\l}'(r_{\l})=  
 		\dfrac{16 \sqrt[4]{3 {\bar\l}^3} V'(1,0)}
 		{\sqrt{\pi^3}}  \big(1+o(1)\big)
 		\sqrt{\l r_{\l}^2 - \frac{\pi^2}4}\ \ \ \hbox{ as $\l\to \bar \l$}.
 	\end{equation}
 \end{lemma}	
 \begin{proof}
 	To get \eqref{norma-N=3} and \eqref{u-prima zona-N=3}
 	it suffices to apply Theorem \ref{thm:BP} to the function $v_\e$ defined in \eqref{e-l}. Estimate \eqref{stima-der-prima-N=3} is an easy consequence of \eqref{u-prima zona-N=3}.
 \end{proof}
 In higher dimension, instead, we have
 \begin{lemma}\label{lem-preliminare-N=4,5,6}
 	Let  $u_\l$ be any radial nodal solution to \eqref{1}  in dimension $N$ between $4$ and $6$,  $r_{\l}$ its first zero in $(0,1)$ and $\bar \l>0$ such that  $\norm{u_\l}_{\infty}=u_\l(0)\to \infty$ when $\l\to \bar\l$.
 	Then we have
 	\begin{equation}\label{norma}\begin{cases}
 			\log \|u_\l\|_\infty
 			= \frac2{\bar \l} (1+o(1)) \, r_\l^{-2}  & \text{ when }N=4, \\
 			\norm{u_\l}_{\infty}= {15}^\frac 34 \left(\frac{24}{\pi\bar\l}\right)^{\frac32}
 			(1+o(1)) \, r_{\l}^{-\frac92} \quad & \text{ when }N= 5, \\
 			\norm{u_\l}_{\infty}=  \frac{1152}{\bar \l}  (1+o(1)) \, r_{\l}^{-4} & \text{ when }N= 6 
 	\end{cases}\end{equation}
 	\begin{equation}\label{stima-der-prima}
 		u_{\l}'(r_{\l})=
 		\begin{cases} - 16 (1+o(1))  \left(\norm{u_\l}_{\infty}r_\l^{3}\right)^{-1}  &\hbox{if }N=4 ,\\
 			-3^{\frac14} {5}^\frac 34 \left(\frac{\pi\bar\l}{8}\right)^{\frac32}
 			(1+o(1)) \, r_{\l}^\frac12 \quad &\hbox{if }N= 5\\
 			- 2\bar \l (1+o(1)) \, r_{\l}^{-1} &\hbox{if }N=6,
 		\end{cases}
 	\end{equation} 
 	for $\l\to\bar \l$.
 \end{lemma}	
 \begin{proof} 
 	We let $\e=\l r_\l^2$ and $v_\e=r_\l^{\frac{N-2}2} u_\l(r_\l r)$ be as in \eqref{e-l}. It solves \eqref{BN-pos} and so it
 	is a positive solution to \eqref{1} corresponding to $\e$.  Moreover by \eqref{r-0} we know that $r_\l\to 0$
 	so that $\e\to 0$. We can then apply Theorem \ref{thm:Han} to $v_\e$ getting that, as $\l\to \bar \l$ 
 	\[ \|u_\l\|_{\infty}=r_\l^{-\frac{N-2}2}\|v_\e\|_{\infty} =
 	C_N r_\l^{-\frac{N-2}2}\left(\l r_\l^2\right)^{-\frac{N-2}{2(N-4)}}(1+o(1)) \  \text{ if } N=5,6 ,
 	\]
 	where $C_N$ is the constant in \eqref{norma-pos-N=4,5,6}. 
 	When $N=4$ instead \eqref{norma-pos-N=4,5,6} gives
 	\[\l r_\l^2 \log \|u_\l\|_\infty+\l r_\l^2\log r_\l\to  2
 	\]
 	as $\l\to \bar \l$, 
 	and \eqref{norma} follows  recalling that $r_\l\to 0$.  
 	Further by \eqref{u-pos-N>3}
 	\[
 	\norm{v_\e}_\infty v_\e(x)\to  [N(N-2)]^{\frac{N-2}2}
 	(N-2) \sigma_N  \, G(x,0) \quad \text{  in } C^1_{\loc}(\bar B\setminus\{0\}),  \]
 	as $\e\to 0$.
 	In particular, using also \eqref{norma-pos-N=4,5,6}	
 	\begin{equation}\label{passaggio}
 		\begin{split}
 			v_\e'(1) &\sim \frac{  [N(N-2)]^{\frac{N-2}2}
 				(N-2) \sigma_N}{\norm{v_\e}_{\infty}}\sum_{i=1}^N x_i \frac{\partial G(x,0)}{\partial x_i}{\Big |_{|x|=1}}\\
 			&=
 			\begin{cases}
 				\frac {  [N(N-2)]^{\frac{N-2}2}(N-2)\sigma_N }{C_N} \e^{ \frac{N-2} {2(N-4)}} \sum_{i=1}^N x_i \frac{\partial G(x,0)}{\partial x_i}{\Big |_{|x|=1}} 
 				&\hbox{ if }N\ge5\\
 				\frac{
 					16\sigma_4}{\norm{v_\e}_{\infty}}\sum_{i=1}^N x_i \frac{\partial G(x,0)}{\partial x_i}{\Big |_{|x|=1}} 
 				&\hbox{ if }N=4.
 			\end{cases}
 		\end{split}
 	\end{equation}
 	Next 	recalling that $u'_\l(r_\l)=r_\l^{-\frac N2}v_\e'(1)$ with $\e=\l r_\l^2$ 
 	we have 
 	\[u'_\l(r_\l)=r_\l^{-\frac N2} v_\e'(1) \sim \begin{cases}
 		\frac {  [N(N-2)]^{\frac{N-2}2}
 			(N-2)\sigma_N }{C_N} \l^{ \frac{N-2} {2(N-4)}} r_\l^{\frac{-N^2+6N-4}{2(N-4)}} \sum_{i=1}^N x_i \frac{\partial G(x,0)}{\partial x_i}{\Big |_{|x|=1}}
 		&\hbox{ if }N\ge5\\
 		\frac{16\sigma_4}{\norm{u_\l}_{\infty}}r_\l^{-3}
 		\sum_{i=1}^N x_i \frac{\partial G(x,0)}{\partial x_i}{\Big |_{|x|=1}}
 		&\hbox{ if }N=4.
 	\end{cases}
 	\]
 	which gives \eqref{stima-der-prima} using the explicit value for $\sum_{i=1}^N x_i \frac{\partial G(x,0)}{\partial x_i}{\Big |_{|x|=1}}=-\frac 1{\sigma_N}$, (see \cite{GT}).
 \end{proof} 
 
 To complete the parts of the proofs of Theorems \ref{teo:N=3} and \ref{teo:N=4,5} concerning the first nodal zone $(0,r_\l)$, it is needed to describe the asymptotics of the first zero $r_\l$ in terms of $\l-\bar\l$. In this matter a role is played by the  behaviour of the solution in the subsequent nodal zones.
 A technical lemma is needed to go further.
 We state it for any zero of the solution.
 \begin{lemma}\label{l2}
 	Let $u_\l$ be any radial nodal solution to \eqref{1}  and $r_{i,\l}$ one of its nodal radii. 
 	For every $0< a \le  r \le  b \le 1$ and  we have 
 	\begin{align}\label{der-prima}
 		u_\l'(r)  =\frac 1{r^{N-1}}\left(b^{N-1}u_\l'(b) +\int_r^b s^{N-1} \left( |u_\l |^{2^*-2}u_\l+\l u_\l
 		\right)\, ds\right), \\
 		\label{stima-norma}
 		u_\l(a)  =\frac{b^{N-1}}{N-2}u_\l'(b)\left(\frac{1}{r_{i,\l}^{N-2}}-\frac{1}{a^{N-2}}\right) +\frac 1{N-2}\Big[\int_a^{r_{i,\l}}r \left( |u_\l |^{2^*-2}u_\l+\l u_\l
 		\right) \, dr\\ \nonumber
 		+\frac{1}{r_{i,\l}^{N-2}}\int_{r_{i,\l}} ^br^{N-1}  \left( |u_\l |^{2^*-2}u_\l+\l u_\l
 		\right) \, dr -\frac{1}{a^{N-2}}\int_a^b r^{N-1}  \left( |u_\l |^{2^*-2}u_\l+\l u_\l
 		\right) \, dr\Big].
 	\end{align}
 \end{lemma}
 \begin{proof}
 	Integrating the equation in \eqref{radial} over $(r,b)$ gives \eqref{der-prima}. Integrating again  over $(a,r_{i,\l})$ then we get 
 	\[
 	-u_\l(a)=\int_a^{r_{i,\l}}\frac 1{r^{N-1}}\left(b^{N-1}u_\l'(b)+\int_r^b s^{N-1}\left(|u_\l|^{2^*-2}t+\l u_\l\right) ds\right) \ dr.\]
 	By simple computations it follows that
 	\begin{align*}
 		u_\l(a)& =\frac{b^{N-1}}{N-2}u_\l'(b)\left(r_{i,\l}^{2-N}-a^{2-N}\right)
 		-\int_a^{r_{i,\l}}r^{1-N}\, \int_r^bs^{N-1} \left(|u_\l|^{2^*-2}t+\l u_\l\right) ds\,dr
 		\intertext{and integrating by parts}
 		& =\frac{b^{N-1}}{N-2}u_\l'(b)\left(r_{i,\l}^{2-N}-a^{2-N}\right)
 		+\frac 1{N-2}\Bigg[\int_a^{r_{i,\l}}r \left(|u_\l|^{2^*-2}t+\l u_\l\right)dr\\ & \ + \frac{1}{r_{i,\l}^{N-2}}\int_{r_{i,\l}} ^br^{N-1}\left(|u_\l|^{2^*-2}t+\l u_\l\right) dr 
 		-\frac{1}{a^{N-2}}\int_a^b r^{N-1}\left(|u_\l|^{2^*-2}t+\l u_\l\right)dr\Bigg]
 	\end{align*}
 	which gives \eqref{stima-norma}.
 \end{proof}
 
 An immediate, but interesting, consequence of Lemma \ref{l2} is that the behaviour of the solutions in the annulus $A_\l$, of radii $r_\l$ and $1$, is controlled by the derivative of $u_\l$ in the first node $r_\l$, that is
 
 \begin{corollary}\label{cor:l2}
 	Let $u_\l$ be any radial nodal solution to \eqref{1}, $r_{\l}$ its first zero in $(0,1)$ and $A_\l$ the annulus of radii $r_\l$ and $1$. 
 	Then
 	\begin{equation}\label{est-norma-altre-zone}
 		\|u_\l\|_{L^{\infty}(A_\l)}	\le -\frac{1}{N-2} r_\l u_\l'(r_\l). 
 	\end{equation}
 \end{corollary}
 \begin{proof}
 	We denote by $s_\l$ the point at which $u_\l$ attains the $L^\infty$-norm in $A_\l$, that is the first minimum of $u_\l$ (see \eqref{remark-vecchio}).
 	So writing \eqref{stima-norma} with  $a=b=s_{\l}$ and $r_{i,\l}=r_\l$ gives
 	\begin{align}
 		\|u_\l\|_{L^{\infty}(A_\l)} = -u_\l(s_\l)= & \frac{1}{N-2}\Big[-\frac{1}{r_{\l}^{N-2}}\int_{r_{\l}} ^{s_{\l}}r^{N-1} \left(|u_\l |^{2^*-2}u_\l+\l u_\l\right) \, dr  \nonumber \\
 		&+ \int_{r_{\l}} ^{s_{\l}} r \left(|u_\l |^{2^*-2}u_\l+\l u_\l\right)\, dr\Big]\label{eq:dipassaggio}
 		\intertext{and noticing that $|u_\l|^{2^*-2} u_\l+ \l u_\l <0$ on $(r_\l, s_\l)$ yields }
 		& \le -\frac{1}{(N-2) r_{\l}^{N-2}}\int_{r_{\l}} ^{s_{\l}}r^{N-1} \left(|u_\l |^{2^*-2}u_\l+\l u_\l\right)\  dr .\nonumber
 	\end{align}
 	On the other hand, choosing $r=r_\l$ and $b=s_\l$ in \eqref{der-prima} gives 
 	\[ u_\l'(r_\l) = \frac{1}{r_\l^{N-1}} \int_{r_{\l}} ^{s_{\l}}r^{N-1} \left( |u_\l |^{2^*-2}u_\l+\l u_\l\right) \,  dr ,\]  
 	which concludes the proof.
 \end{proof}
 
 Corollary \ref{cor:l2}, together with Lemmas \ref{lem-preliminare-N=3} and \ref{lem-preliminare-N=4,5,6} furnishes an estimate of the norm of $u_\l$ on the set $A_\l$, which shows that in dimension $N$ between $3$ and $6$ the solution does not behave like a tower of bubbles.
 
 \begin{lemma}\label{lem:altre-zone-sup}
 	Let  $u_\l$ be any radial nodal solution to \eqref{1} in dimension $N$ between $3$ and $6$, and $\bar \l>0$ such that $ \norm{u_\l}_{L^\infty(B)}=u_\l(0)\to\infty$ as $\l\to\bar\l$, $r_{\l}$ its first zero in $(0,1)$ and $A_\l$ the annulus of radii $r_\l$ and $1$. Then 
 	\begin{align}\label{norma-A}
 		\|u_\l\|_{L^{\infty}(A_\l)} & \to  0 \quad & \text{ as $\l\to \bar \l$, \quad if $N=3,4$ or $5$}, \\
 		\|u_\l\|_{L^{\infty}(A_\l)} & \le C \quad & \text{ if $N=6$}.\label{norma-A-6}
 	\end{align}
 \end{lemma}
 \begin{proof}
 	Inserting \eqref{stima-der-prima-N=3}, \eqref{stima-der-prima} into \eqref{est-norma-altre-zone} gives  
 	\begin{align*}
 		0\le \|u_\l\|_{L^{\infty}(A_\l)} & \le - \frac{1}{N-2} r_\l u_\l'(r_\l) \sim  \begin{cases} 
 			C \sqrt{\l r_{\l}^2 - \pi^2/4} & \text{ if } N=3,  \\
 			C  (r_\l^2\|u_\l\|_{\infty})^{-1} & \hbox{if } N=4,
 			\\
 			C r_{\l}^{\frac{3}{2}}&\hbox{if } N = 5,\\
 			C &\hbox{if } N = 6
 		\end{cases}
 	\end{align*}
 	as $\l\to \bar \l$. So the claim readily follows by  \eqref{r-0} if $N=3$ or $5$.
 	Otherwise if $N=4$,  \eqref{norma} yields $r_\l^2\|u_\l\|_{\infty} = r_\l^2 e^{\frac{2+o(1)}{\bar \l r_\l^2}}$, and   \eqref{r-0} allows to conclude also in this case.
 \end{proof}
 We end this section with a uniqueness result for  solutions to \eqref{1}. It will be used in a crucial way in Section \ref{s4}.
 \begin{proposition}\label{un} 
 	Let $u_1$ and $u_2$ be radial solutions to \eqref{1} corresponding to $\l_1$ and  $\l_2$ respectively. If $u_1$ and $u_2$ have both $m$ nodal zones and
 	\begin{equation}\label{b20}
 		\frac{u_1(0)}{\l_1^{\frac{N-2}{4}}}=\frac{u_2(0)}{\l_2^{\frac{N-2}{4}}} ,
 	\end{equation}
 	then  
 	\begin{equation}\label{uniq}
 	 \l_1=\l_2 \text{  and 	} u_1\equiv u_2.
 	\end{equation}
 \end{proposition}
 \begin{proof}
 	Let us consider the functions $\tilde u_1:\left(0, \sqrt{\frac{\l_1}{\l_2}} \right)\to\R$ as
 	\begin{equation}
 		\tilde u_1(r)= \left(\frac{\l_2}{\l_1}\right)^{\frac{N-2}{4}} u_1\left(\sqrt{\frac{\l_2}{\l_1}}r\right)
 	\end{equation}
 	which verifies
$$
\begin{cases} 
-\tilde u_1''-\frac{N-1}r\tilde u_1'=|\tilde u_1|\tilde u_1+\l_2\tilde u_1 & \text{ in } \ \left(0,\sqrt{\frac{\l_2}{\l_1}}\right) , \\
\tilde u_1(0)=\left(\frac{\l_2}{\l_1}\right)^{\frac{N-2}{4}} u_1(0) , \\
\tilde u_1'(0)=0 .\\
\end{cases}
$$
So if \eqref{b20} holds then $\tilde u_1\equiv u_2$  by the uniqueness of the solution to the Cauchy problem. In particular the $m^{th}$ zero of $u_2$, that is $1$, coincides with the $m^{th}$ zero of $\tilde u_1$, that is $\sqrt{\frac{\l_1}{\l_2}}$, and the claim follows.
 \end{proof}
 
 \begin{remark}
 	Proposition \ref{un} is equivalent to the uniqueness of the solution of the transformed (with the Emden-Fowler transformation) problem in \cite[Section 2]{ABP90} with a fixed asymptotic value at infinity. 
 \end{remark}

\begin{remark}\label{rem:uniq}
In dimension $N=6$, Proposition \ref{un} states that the overdetermined problem \eqref{i10} is fulfilled by one couple $(\bar \l, u^{m-1})$ at most. This allows us to characterize both the  concentration value $\bar \l$ and the asymptotic profile of the  solution outside the origin. It will be a key ingredient in the proof of Theorem \ref{teo:N=6}.
 \end{remark}

\section{Proof of Theorems  \ref{teo:N=3} and \ref{teo:N=4,5}}\label{s3}
In this section we compute the rate of $r_\l$, $\|u_\l\|_{L^{\infty}(B)}$ and  $\|u_\l\|_{L^{\infty}(A_\l)}$ in terms of $\l-\bar\l$, and describe the profile of the solution in the set $A_\l$ for $N=3,4,5$. Note that we characterize the value of $\bar\l$ by the asymptotic behaviour of the solution $u_\l$ in the annulus $A_\l$. This is a completely different approach to the other proof in the literature (see \cite{ABP90}, \cite{GG} or \cite{IP} for example).

We argue separately according to the dimension.

\subsection{The case $N=3$}\label{subsec:N=3}~\\

Let us give a first  description of  the profile of the solution in the annulus $A_\l$ of radii $r_\l$ and $1$.

\begin{proposition}\label{prop:A-N=3} 
	Let  $u_\l$ be any radial nodal solution to \eqref{1} with $m\ge 2$ nodal zones in dimension $N=3$, and $\bar \l$ such that 
	$\|u_\l\|_{\infty}=u_\l(0) \to \infty$ when $\l\to\bar \l$.
	Then, 
	\begin{align}\label{lambda-star-m}
		\bar \l &= \frac{(2m-1)^2}{4}\pi^2 , \qquad r_\l\to \bar r = \frac{1}{2m-1}, \\
		\label{u-star-m-N=3}
		\dfrac{u_{\l}( x) }{\norm{u_\l}_{L^{\infty}(A_\l)}} & \to - \frac{2\theta_o}{(2m-1)\pi\cos\theta_o}\frac{1}{|x|} \cos\left(\frac{2m-1}{2}\pi |x|\right) \quad \text{ in }  C^1_{\loc}\left(\bar B_1 \setminus \bar B_{\frac{1}{ 2m-1}} \right)
		\intertext{ \text{ as $\l\to \bar \l$} where $ \theta_o\approx 2,\!7984$ is the unique root of $1+\theta\tan\theta=0$ in the interval $(\frac {\pi}2,\frac 3 2\pi)$.
			Furthermore $\norm{u_\l}_{L^{\infty}(A_\l)} \to 0$, more precisely}
		\label{norma-A-N=3-prel}
		\norm{u_\l}_{L^{\infty}(A_\l)} & =
		4 \frac{\sqrt[4]3 V'(1,0) \sqrt{2(2m-1)} \cos\theta_o}{\theta_o}  (1+o(1))  \sqrt{\l r_\l^2 - \pi^2/4}   , 
	\end{align}
	where $V(r,0)$ is the function defined in \eqref{green-pi}.
\end{proposition}

\begin{proof}
	Set $M_\l=\|u_\l\|_{L^{\infty}(A_\l)}$ 
	and look at the normalized function 
	\begin{equation}\label{risc-altre-zone-N=3} 
		\widetilde u_\l(x):= \dfrac{ u_{\l}( x) }{M_\l} 
	\end{equation} 
	which  solves
	\begin{equation}\label{eq-v-tilde-N=3}\begin{cases}
			-\Delta u = (M_\l)^4 |u|^{4} u+   \l u  \quad & \text{ in } A_\l, \\
			u=0 & 	\text{ on } \partial A_\l .\end{cases}\end{equation}  
	By construction $|\widetilde u_\l|\le 1$ on $(r_\l,1)$. We denote by $s_\l\in (r_\l,1)$ the point at which the $L^\infty$-norm of $u_\l$ in $A_\l$ is attained, i.e. the first minimum of $u_\l$ according to \eqref{remark-vecchio}.
	Integrating the equation in \eqref{eq-v-tilde-N=3} (written in radial coordinates)  gives that, for every $r\in [r_\l, 1]$,
	\begin{equation}\label{aux-1}
		\left|(\widetilde u_\l)'(r)\right|=\frac 1{r^{2}} \left|\int_r^{s_\l}s^{2}\left( (M_\l)^4 |\widetilde u_\l|^{4}\widetilde u_\l+\l \widetilde u_\l\right)\, ds\right| \le  C 
	\end{equation}
	since $r_\l\to \bar r = \frac {\pi}{ 2 \sqrt {\bar \l}}$ by
	\eqref{r-0} and $M_\l\to 0$
	by \eqref{norma-A}. 
	Then \eqref{eq-v-tilde-N=3} gives
	\[ |(\widetilde u_\l)''|\le\frac{2|(\widetilde u_\l)'|}r+ (M_\l)^4 |\widetilde u_\l|^5+\l|\widetilde u_\l|\le C ,\]
	for every $r>\bar r >0$
	and the Ascoli-Arzel\'a Theorem yields that $\widetilde u_\l\to w$ in $C^1[\delta,1]$, for every $\d > \bar r$. 
	The limit function $w$ is radial and solves
	\begin{equation}\label{eq:autof-anello} 
		\begin{cases}-\Delta w = \bar \l w  \quad & \text{ in } A^* , \\
			\| w\|_\infty=1\\
			w= 0 & \text{ on } \partial A^*,	
	\end{cases}\end{equation} 
	where $A^*$ stands for the annulus of radii $\bar r$ and $1$.
	Moreover \eqref{aux-1} ensures that  the minimum point $s_\l$ converges to some point $\bar s > \bar r$, because
	\begin{align*}
		1 & = \widetilde u_\l(r_\l) - \widetilde u_\l(s_\l)= -\int_{r_\l}^{s_\l} (\widetilde u_\l)'(r)\, dr 
		\le C(s_\l - r_\l).
	\end{align*}
	Hence $w$ is nontrivial since $w(\bar s)=-1$. Further $w<0$ on $(\bar r, \bar s)$: indeed it is clear that $w\le 0$ on $(\bar r, \bar s)$, and if by contradiction $w(t)=0$ at some point $t\in(\bar r, \bar s)$, then also $w'(t)=0$ because of the sign condition, implying $w\equiv 0$ (since $w$ solves the ODE obtained by writing \eqref{eq:autof-anello} in radial coordinates).
	Similarly, 
	one can check that $w$ has exactly $m-1$ nodal zones, that is $w$ is the $(m-1)^{th}$ radial eigenfunction of the Laplacian in the annulus $A^*$. 
	Indeed, assume that a nodal zone $(r_{1,\l},r_{2,\l})\subset (r_\l,1)$ disappears as $\l\to \bar \l$, so that $\lim r_{1,\l}=\lim r_{2,\l}=r_0$. Then, $r_0\in [\bar s, 1]$ and 
	there exists a point $ \xi_\l \in (r_{1,\l},r_{2,\l})$ such that $
	\widetilde u_\l'( \xi_\l)=0$. The convergence in $C^1_{\loc}(\bar B_1\setminus \bar B_{\bar r})$ implies that $0=\lim \widetilde u_\l'( \xi_\l)=w'(r_0)=\lim \widetilde u_\l(r_{1,\l})=w(r_0)$ and this is not possible since $w$ solves the ODE obtained by writing \eqref{eq:autof-anello} in radial coordinates as before. 	
	It is easy to see that any radial solution to the equation in \eqref{eq:autof-anello} has the form
	\[ w(x) = \frac{a}{|x|}  \cos(\sqrt{ \bar \l }|x|)+ \frac{b}{|x|} \sin(\sqrt{ \bar \l }|x|) . \]
	for suitable $a$ and $b$.
	Imposing  $w(x)=0$ when $|x|=\bar r $  gives $b=0$ 
	because $\sqrt{\bar\l}  \bar r = \pi/2$ by \eqref{r-0}.
	The additional condition $w(x)=0$ for $|x|=1$ yields that $\sqrt{\bar \l}$  is a zero of the cosine.
	Then $\sqrt{\bar \l} = \frac{2m-1}{2}\pi$ because $w$ has exactly $m-1$ nodal annuli, and  \eqref{r-0} implies that $\bar r=\frac{1}{2m-1}$, concluding the proof of \eqref{lambda-star-m}.
	In particular $ \bar s$ is the minimum point of $\frac{a}{r}  \cos\left(\frac{(2m-1) \pi}{2} r\right)$ on $[\bar r,1]=[1/{(2m-1)}, 1]$ which coincide with the minimum point in the first nodal region of $w$, namely $[1/{(2m-1)}, 3/{(2m-1)}]$ and this implies that $\bar s=\frac {2}{(2m-1)\pi}\theta _o$ where $\theta_o$ is the unique root of $g(\theta)=\theta \tan \theta+1$ in $[\frac{\pi}2,\frac 3 2 \pi]$. 
	Recalling that $w(\bar s)=-1$ a straightforward computation allows us to deduce the exact value of the constant $a$, 
	obtaining  \eqref{u-star-m-N=3}.
	
	{ Lastly from \eqref{eq-v-tilde-N=3} we have 
		\[
		\frac{u'_\l(r_\l)}{M_\l}=\widetilde u'_\l(r_\l)=
		\frac 1{r_\l^2}\int_{r_\l}^{s_\l}s^2 \left((M_\l)^4|
		\widetilde u_\l |^4
		\widetilde u_\l+\l \widetilde u_\l\right) ds
		\]
		and, since $M_\l\to 0$, $r_\l\to \bar r$, $s_\l\to \bar s$ and $\widetilde u_\l\to w$ pointwise with $|\widetilde u_\l|\leq 1$, we 
		can pass to the limit getting
		\[\frac{u'_\l(r_\l)}{M_\l}=\widetilde u'_\l(r_\l)\to \frac 1{\bar r ^2}\bar \l \int_{\bar r}^{\bar s}s^2 w ds=w'(\bar r)=-\dfrac{a\sqrt{\bar\l}}{\bar r}
		\]
		where last equality follows from \eqref{eq:autof-anello} and the fact that $w'(\bar s)=0$.}
	Comparing it with \eqref{stima-der-prima-N=3} we infer 
	\[	\| u_\l\|_{L^{\infty}(A_\l)} = M_\l \sim  \frac{8 \sqrt[4]{3\bar \l} V'(1,0) }{\sqrt \pi \bar r w'(\bar r)}
	\sqrt{\l r_\l^2 - \frac {\pi^2}4} .\]
	Inserting  \eqref{lambda-star-m} and \eqref{u-star-m-N=3} in this last formula gives  \eqref{norma-A-N=3-prel}.
\end{proof}

In Propositions \ref{lem-preliminare-N=3} and \ref{prop:A-N=3} we have depicted the profile of $u_\l$ in the sets $(0,r_\l)$ and $(r_\l,1)$, respectively, in terms of the quantity $\l r_\l^2 - \pi^2/4$ which is infinitesimal as $\l\to \bar \l$. Eventually we complete the proof of Theorem \ref{teo:N=3} by describing the asymptotics of  $\l r_\l^2 - \pi^2/4$ in terms of $\l-\bar \l$.

\begin{proof}[Proof of Theorem \ref{teo:N=3}]
	Equation \eqref{lambda-star-m} states that if $\|u_\l \|\to \infty$ then
	$\l\to \bar \l =  \left(\frac{2m-1}{2}\pi\right)^2$ and $r_\l\to\bar r = \frac{1}{2m-1}$.
	
	We let $\widetilde u_\l$ as in \eqref{risc-altre-zone-N=3} and
		\[ w(x)= - \frac{2\theta_o}{(2m-1)\pi\cos\theta_o}\frac{1}{|x|} \cos\left(\frac{2m-1}{2}\pi |x|\right) \]
		the limit function in \eqref{u-star-m-N=3}, which solves \eqref{eq:autof-anello}. Multiplying the equation in \eqref{eq-v-tilde-N=3} by $w$, the one in \eqref{eq:autof-anello} by $\widetilde u_\l$, integrating by parts on $(r_\l,1)$ and subtracting   gives 
	\begin{align}\label{int-a-croce}
		(\l- \bar \lambda) \int_{r_{\l}}^1 r^{2} \widetilde u_\l w dr + (M_\l)^{4} \int_{r_\l}^1 r^{2} |\widetilde u_\l|^{4}\widetilde u_\l w dr  - r_\l^{2} (\widetilde u_\l)'(r_\l) w(r_\l)  = 0 .
	\end{align}
	Here when $\l\to \bar \l$ we have  
	\begin{align*}
		\int_{r_{\l}}^1 r^{2} \widetilde u_\l w dr \to  \int_{\bar r}^1 r^2 w^2 dr > 0 , \\
		{\int_{r_\l}^1 r^{2} |\widetilde u_\l|^{4}\widetilde u_\l w dr}\to  \int_{\bar r}^1 r^2 w^6 dr > 0, \\
		r_\l^2 \widetilde u_\l'(r_\l) \to  \bar r^2   w'(\bar r),  
		\qquad w(r_\l)  = \left(w'(\bar r)+o(1)\right) (r_\l - \bar r)  
	\end{align*}
	Since $\l r_\l^2-\frac{\pi^2}{4} = r_\l^2 (\l - \bar \l) + \bar\l (r_\l+\bar r) (r_\l - \bar r)$, then 
	\begin{align}\nonumber
		r_\l - \bar r   & = \dfrac1{\bar \l (r_\l+\bar r)} \left[ (\l r_\l^2 - {\pi^2}/{4}) - r_\l^2 (\l - \bar \l )  \right] \\ \label{aux-2}
		&  =  \dfrac1{2\bar \l \bar r} \left(1 + o(1)\right)\left[ (\l r_\l^2 - {\pi^2}/{4})  - \left({\bar r}^2+o(1)\right)  (\l - \bar \l) \right], 
	\end{align}
	so that
	\begin{equation}\label{miserve}
		r_\l^{2} (\widetilde u_\l)'(r_\l) w(r_\l)  =\left( \dfrac{\bar r \left(w'(\bar r)\right)^2}{2\bar \l}+ o(1) \right) (\l r_\l^2 - {\pi^2}/{4})  - \left( \dfrac{\bar r^3 \left(w'(\bar r)\right)^2}{2\bar \l}+o(1)\right)  (\l - \bar \l ) .
	\end{equation}
	Furthermore $(M_\l)^{4}$ is negligible compared to $\l r_\l^2 - \frac{\pi^2}{4}$ by \eqref{norma-A-N=3-prel}.
	Eventually  \eqref{int-a-croce} and \eqref{miserve} imply 
	\[  \left(\int_{\bar r}^1 r^2 w^2 dr  + \dfrac{\bar r^3 \left(w'(\bar r)\right)^2}{2\bar \l} + o(1) \right)  (\l - \bar \l ) = \left( \dfrac{\bar r \left(w'(\bar r)\right)^2}{2\bar \l}+ o(1) \right) \left(\l r_\l^2 - \frac{\pi^2}{4}\right) , \]
	and recalling the explicit form of  $w$ in  \eqref{norma-A-N=3-prel}	we end up with 
	\begin{equation}\label{traduzione-asintotica-N=3}
		\l r_\l^2 - \frac{\pi^2}{4} \sim \frac{4m-3}{(2m-1)^2} (\l - \bar \l ) 
	\end{equation}
	Remembering Lemma \ref{oss-1}, it follows  that $\l \to \bar \l$ from above.
	Next inserting \eqref{traduzione-asintotica-N=3} into \eqref{norma-N=3}, \eqref{u-prima zona-N=3}, \eqref{aux-2} and \eqref{norma-A-N=3-prel} completes the proof of \eqref{norma-B-N=3}, \eqref{r-star} and \eqref{inseire}, respectively. Finally \eqref{u-A-N=3} follows by  \eqref{u-star-m-N=3} and \eqref{traduzione-asintotica-N=3}.
\end{proof}

\subsection{The case $N=4,5$}\label{subsec:N=4,5}~\\
In this section we write $\psi_h$ meaning the $h^{th}$ radial eigenfunction of $-\Delta$ normalized so that
	\[ \psi_h(0)=-1 .\] As in the case $N=3$, next proposition reduces the rate of $\|u_\l\|_{L^{\infty}(A_\l)}$ to that of $r_\l$. Note that the computation below also shows that
\begin{equation}\label{lambda-star-m-N=4,5}
	\bar \l=  \mu_{m-1}
\end{equation}
giving an alternative proof to the same statement in \cite[Theorem B]{ABP90}.
\begin{proposition}\label{prop:A-N=4,5} 
	Let  $u_\l$ be any radial nodal solution to \eqref{1} with $m\ge 2$ nodal zones in dimensions $N=4,5$, $r_\l$ its first zero in $(0,1)$, $A_\l$ be the annulus of radii $r_\l,1$ 
	and $\bar \l$ such that 
	$\|u_\l\|_{\infty}=u_\l(0) \to \infty$ when $\l\to\bar \l$.
	Then 
	\begin{align}
		\label{u-star-m-N=4,5}
		\dfrac{u_{\l}( x) }{\norm{u_\l}_{L^{\infty}(A_\l)}} & \longrightarrow   \psi_{m-1}(x) \qquad \text{ in } C^1_{\loc}\left(\bar{B}\setminus \{0\} \right). 
		\intertext{ 	Furthermore $\norm{u_\l}_{L^{\infty}(A_\l)} \to 0$, more precisely}
		\label{norma-A-N=4,5-prel}
		\norm{u_\l}_{L^{\infty}(A_\l)}
		&=  \begin{cases} 
			8 (1+o(1)) \left(r_\l^2 \|u_\l\|_\infty\right)^{-1} & \text{ if } N=4 , \\
			\left(\frac{5}3\right)^{\frac{3}4} \left(\frac{\pi\mu_{m-1}}{8}\right)^{\frac32}  (1+o(1)) \,    r_\l^{\frac{3}{2}}  \quad & \text{ if } N=5.
		\end{cases}
	\end{align} 
\end{proposition}

\begin{proof}
	We write 
	\[M_\l:= \|u_\l\|_{L^{\infty}(A_\l)}= - u_\l(s_\l), \]
	where $s_\l$ stands for the first minimum of $u_\l$ according to \eqref{remark-vecchio}, and look at the normalized function 
	\begin{equation}\label{risc-altre-zone-N=4} 
		\widetilde u_\l(x) = \begin{cases} u_\l(x)/M_\l \quad & \text{ if } r_\l \le |x| \le 1 , \\ 0 & \text{ if } |x| < r_\l \end{cases}
	\end{equation} 
	which  solves
	\begin{equation}\label{eq-v-tilde-N=4}\begin{cases}
			-\Delta \widetilde u_\l = (M_\l)^{2^*-2} |\widetilde u_\l|^{2^*-2} \widetilde u_\l +   \l \widetilde u_\l   \quad & \text{ in } A_\l, \\
			|\widetilde u_\l |\leq 1& \text{ in } A_\l, \\
			\widetilde u_\l=0 & 	\text{ on } \partial A_\l .\end{cases}\end{equation} 
	The function $\widetilde u_\l$ is uniformly bounded both in $L^\infty(B)$ and in $H^1_{0,\rad}(B)$ since we know by \eqref{norma-A} that $M_\l\to 0$ as $\l \to \bar \l$. Remembering that also $r_\l \to 0$ by \eqref{r-0}, it is easy to see that  $\widetilde u_\l  $ converges weakly in $H^1_{0,\rad}(B)$ to a radial bounded function $\psi\in H^1_{0,\rad}(B)$ which is a weak solution to 
	\begin{equation}\label{eq-autofunz}\begin{cases} -\Delta \psi = \bar \lambda \psi \quad & \text{ in } B, \\
			\psi=0 & \text{ on } \partial B. \end{cases}\end{equation}
	Therefore $\psi(x)=  {A}\psi_n(x)$ 
	for some integer $n$ and for some constant $A$, which can possibly be equal to 0. 
	Notice that  
	\begin{equation}\label{stima-der-prima*}
		|\widetilde u_\l'(r)|\leq \frac 1{r^{N-1}}\left| \int_{s_\l}^r t^{N-1} \left((M_\l)^{2^*-2} |\widetilde u_\l|^{2^*-2} \widetilde u_\l +   \l \widetilde u_\l \right)dt \right|
		\leq Cr \left|\frac{s_\l^N}{r^N}-1\right|
	\end{equation}
	so that for every $r>0$, $|\widetilde u_\l'(r)|\leq C$ and the Ascoli-Arzelà Theorem yields that $\widetilde u_\l$ converges to $\psi$ in $C_{\loc}(B\setminus \{0\})$. In the same manner \eqref{eq-v-tilde-N=4} gives 
	\[|\widetilde u_\l (r)''|\leq \frac Cr |\widetilde u_\l'(r)|+C\]
	and the convergence holds in $C^1_{\loc}(B\setminus \{0\})$. \\
	Next we claim that $s_\l\to 0$. Indeed if, instead $\bar s= \lim s_\l >0$ the uniform convergence implies that $\psi(\bar s)=-1$ so that $A\neq 0$ and $\bar s>0$ is the global maximum of $|\psi|$ in $(0,1)$ 
	and this contradicts the fact that any nontrivial radial solution to \eqref{eq-autofunz} has global maximum  (or minimum) in the origin.\\
	Hence for every $r>0$  we may assume that $s_\l<r$, so that 
	\begin{align*}
		\widetilde u_\l(r) = & \widetilde u_\l(s_\l) + \int_{s_\l}^r \widetilde u_\l'(t) dt = -1 + \int_{s_\l}^r \widetilde u_\l'(t) dt.
	\end{align*}
	On the other hand estimate \eqref{stima-der-prima*}  gives
	\begin{align}\label{passaggio-45}
		-1\le  \widetilde u_\l(r) \le  &-1 + 
		C \int_{s_\l}^r  t dt \le -1 + 
		\frac C2r^2 ,
	\end{align}
	since $r>s_\l$, which with the pointwise convergence gives $\psi(0)=-1$, implying at once that $\psi$ is not trivial and $A=1 $.
	Further $\psi$ admits $m-1$ nodal regions. Indeed \eqref{passaggio-45} ensures also that the first nodal zone do not collapse to $r=0$, since, denoting by $\tilde r_\l$ the first zero of $ \widetilde u_\l$ in $(r_\l,1)$ we have $0= \widetilde u_\l(\tilde r_\l)\leq -1+C(\tilde r_\l^2-s_\l ^2)$. Next, the $C^1_{\loc}$ convergence in $B\setminus \{0\}$ and the Rolle Theorem imply that any nodal zone of $ \widetilde u_\l$ cannot disappear as one can see repeating the proof in the case when $N=3$. Then $\bar \l=\mu_{m-1}$, $n=m-1$ and the proof of 
	\eqref{lambda-star-m-N=4,5}, \eqref{u-star-m-N=4,5} is completed.\\
	Finally we prove \eqref{norma-A-N=4,5-prel}. By \eqref{eq:dipassaggio} we have that 
	\[M_\l=-u_\l(s_\l)=-\frac 1{N-2}r_\l u'(r_\l)+\frac 1{N-2}\int_{r_\l}^{s_\l} r \left(|u_\l|^{2^*-2}u_\l+\l u_\l\right)dr\]
	and, using $M_\l\to 0$ 
	\[
	\left|\int_{r_\l}^{s_\l} r \left(|u_\l|^{2^*-2}u_\l+\l u_\l\right)dr \right|\leq C s_\l^2 M_\l=o(M_\l)\]
	since we have already proved that $s_\l \to 0$ as $\l\to  \bar \l$. Then
	\[ M_\l\sim  -\frac 1{N-2}r_\l u'(r_\l) \]
	from which \eqref{norma-A-N=4,5-prel} follows recalling the behaviour of $u'_\l(r_\l)$ in \eqref{stima-der-prima}. 
\end{proof}

\begin{proof}[Proof of Theorem \ref{teo:N=4,5}]
	
	From Proposition \ref{prop:A-N=4,5} follows that
	$\| u_\l \|_{\infty}\to \infty$ can happen if and only if $\l\to \mu_{m-1}$ and implies that $r_\l\to 0$.
	To get the other estimates 
	we depict first the asymptotics of $r_\l$ in terms of the quantity $\mu_{m-1}-\l$. We consider the function $\widetilde u_\l$ defined in \eqref{risc-altre-zone-N=4} which solves \eqref{eq-v-tilde-N=4} where $A_\l$ is the annulus of radii $r_\l$ and $1$ and $M_\l=\|u_\l\|_{L^\infty(A_\l)}$. 
	Multiplying equation \eqref{eq-v-tilde-N=4}
	by $\psi_{m-1}$, the one in \eqref{eq-autofunz} by $\widetilde u_\l$, integrating by parts on $(r_\l,1)$ and 
	subtracting gives
	\begin{align}\label{passo-passo}
		(\mu_{m-1}-\l) \int_{r_{\l}}^1 r^{N-1} \widetilde u_\l \psi_{m-1} dr = & (M_\l)^{2^*-2} \int_{r_\l}^1 r^{N-1} |\widetilde u_\l|^{2^*-2}\widetilde u_\l \psi_{m-1} dr \\
		& - r_\l^{N-1} \widetilde u_\l'(r_\l) \psi_{m-1}(r_\l) .\nonumber
	\end{align}
	Next, we observe that by \eqref{u-star-m-N=4,5} 
	\begin{align*}
		\int_{r_{\l}}^1 r^{N-1} \widetilde u_\l \psi_{m-1} dr \to\int_0^1 r^{N-1} (\psi_{m-1})^2 dr =A_1>0 \\
		\int_{r_\l}^1 r^{N-1} |\widetilde u_\l|^{2^*-2}\widetilde u_\l \psi_{m-1} dr\to \int_0^1 r^{N-1} |\psi_{m-1}|^{2^*} dr =A_2>0
	\end{align*}
	and $\psi_{m-1}(r_\l)\to -1$ since $r_\l\to 0$ by \eqref{r-0}. Dividing by $\int_{r_{\l}}^1 r^{N-1} \widetilde u_\l \psi_{m-1} dr $  \eqref{passo-passo} gives 
	\begin{align}\label{pp}
		\mu_{m-1}-\l = & \left(\frac{A_2}{A_1}+o(1)\right) (M_\l)^{2^*-2} + \left(\frac 1{A_1}+o(1)\right) r_\l^{N-1} \widetilde u_\l'(r_\l) .
	\end{align}
	Moreover it is known by \eqref{norma-A-N=4,5-prel} and \eqref{stima-der-prima} that, as $\l\to \mu_{m-1}$
	\[
	(M_\l)^{2^*-2} \sim \begin{cases} 
		\frac{5}{3} \left(\frac{\pi\mu_{m\!-\!1}}{8} \right)^2 r_\l^2& \text{ if } N=5, \\  
		\frac{64}{r_\l^4\|u_\l\|_\infty^2} & \text{ if } N=4,\end{cases} \qquad 
	r_\l^{N-1} \widetilde u_\l'(r_\l)  \sim \begin{cases} 
		-3 r_\l^3 & \text{ if } N=5, \\ 
		-\frac {16}{\|u_\l\|_\infty}  & \text{ if } N=4,\end{cases}
	\]
	Therefore, depending on the dimension, the first or the second term in \eqref{pp} dominates. Eventually as $\l\to \mu_{m-1}$
	\begin{equation}\label{ALnum}
		\mu_{m-1}-\l  \sim   \begin{cases}  \frac {5 A_2}{3 A_1} \left(\frac{\pi\mu_{m\!-\!1}}{8}\right)^2  r_\l^2& \text{ if } N=5, \\  
			-\frac { 16} {A_1 \|u_\l\|_\infty}
			& \text{ if } N=4,  \end{cases}
	\end{equation}
	showing that $\l$ goes to $\mu_{m-1}$ from below in dimension $N=5$, and from above in dimension $N=4$.
	In dimension $N=5$, inverting this relation  gives \eqref{r-N=4,5}, next \eqref{r-N=5} and  \eqref{u-A-N=4} easily follow from \eqref{norma} and from \eqref{u-star-m-N=4,5}, \eqref{norma-A-N=4,5-prel}, respectively.
	In dimension $N=4$, instead, inverting \eqref{ALnum} provides  \eqref{r-N=5}.
	Moreover by \eqref{norma} we have 
	\[\log (\l-\mu_{m-1})\sim -\log \|u_\l\|_\infty\sim -2r_\l^{-2}\mu_{m-1}^{-1} , 
	\]
	from which follows
	\[r_\l\sim \sqrt{\frac {-2}{\mu_{m-1}\log (\l-\mu_{m-1})}}, \]
	that is \eqref{r-N=5}. Eventually inserting \eqref{r-N=5} into \eqref{u-star-m-N=4,5}, \eqref{norma-A-N=4,5-prel} gives \eqref{u-A-N=4}.
\end{proof}

 \section{The case $N=6$}\label{subsec:N=6}\label{s4}

In this section we consider the most delicate case $N=6$. As mentioned in the Introduction, the strategy of the proof  of the cases $N=3,4,5$ here does not work. Indeed,
although the blow-up estimates of Section \ref{s2} hold, the integral identities \eqref{int-a-croce} and \eqref{passo-passo}  allow to identify the concentration value $\bar \l$ and the limiting profile in \eqref{u-A-N=6}, but not the rates of the relevant quantities. We start with a result that proves \eqref{norma-A-N=6} and \eqref{u-A-N=6} and then we will give the details of the proof of main part of this section.  Another proof of this first result can be found in \cite[Theorem 2]{ABP90} in terms of the Emden-Fowler transformation. This is an alternative proof without using this transformation.
 \begin{lemma}[Proof of  \eqref{norma-A-N=6} and \eqref{u-A-N=6}  of Theorem \ref{teo:N=6}]\label{lem-6-1}
		Let  $u_\l$ be any radial solution to \eqref{1} with $m\ge 2$ nodal zones in dimension $N=6$, and assume that $\|u_\l\|_{\infty}=u_\l(0) \to \infty$ when $\l\to\bar \l>0$. Then 
		\begin{equation}\label{u-star-m-N=6}
		u_\l\to u^{m-1} \qquad \text{ in } C^1_{\loc}\left(\bar{B}\setminus \{0\} \right) 
		\end{equation}
		where $u^{m-1}$ is the unique radial solution to \eqref{i10} and $\bar \l$ is the unique value at which \eqref{i10} admits a radial solution. 
	\end{lemma}
	\begin{proof}
		By \eqref{norma-A-6} we know that the function $u_\l$ is bounded in $(r_{\l},1)$ and satisfies $u_\l(r_{\l})=0$. We extend $u_\l$ to zero in $(0,r_\l)$ so that, denoting by  $\widehat u_\l$ the extended functions, by \eqref{1} we have
		\[\int_0^1|\nabla \widehat u_\l|^2 dx=\int_{r_{\l}}^1 r^{5} (u_{\l}')^2 dr =\l \int _{r_{\l}}^1 r^5 (u_\l)^2 dr+\int_{r_{\l}}^1 r^5(u_\l)^2 dr\leq C.\]
		Then $\widehat u_\l$ converges to a function $u_*\in H^1_0(B)$ as $\l\to \bar \l$, up to a subsequence. 
		The convergence is weak in $H^1_0(B)$ and a.e. in $B$. Next, again by \eqref{1}, we have that
		\begin{align}\label{eq:sol-int}
		\widehat u_\l(r)= & \int_{r}^1 s^{-5}\int _{s_\l}^s t^5 \left(\l u_\l+|u_\l| u_\l\right) dt ds , \\
		\label{fra-nuova1}
		(\widehat u_\l)'(r)=& -\frac 1{r^5}\int _{s_\l}^r t^5 \left(\l u_\l+|u_\l| u_\l\right) dt, \end{align}
		for every $r>r_\l$.
		 The boundedness of $u_\l$ in $(s_{\l},1)$ shows that 
\begin{equation}\label{ff1}		 
	|(\widehat u_\l)'(r)|\leq C\Leftrightarrow|u_\l'(r)|\leq C\hbox{ for any }r>s_\l
\end{equation}	 
		 and  the pointwise convergence of $\widehat u_\l$ jointly with $r_{\l}\to 0$ as $\l\to \bar \l$ allows to pass to the limit in the previous integrals implying that the convergence of $\widehat u_\l$ to $u_*$ holds in $C^1_{\loc}\left(\bar{B}\setminus \{0\} \right)$.  \\ 	Furthermore the weak convergence in $H^1_0(B)$ of $\widehat u_\l$ to  $u_*$ gives
		 \[\int_B \nabla  u_* \cdot \nabla \psi dx=\bar \l\int_B u_* \psi dx+\int_B  | u_*|u_* \psi dx\]
		 for every $\psi\in C^1_0(B)$.  The pointwise convergence then implies that $|u_*|\leq C$, $|(u_*)'|\leq C$ and that $u^*$ is a classical solution to $-\Delta u^*=\bar \l u^*+|u^*|u^*$ in $B$. 
		Next we show that $M_\l\to M>0$ as $\l\to \bar \l$. 
		Using \eqref{eq:dipassaggio} we have
		\begin{align}\label{eq:stima-norma-2}
		M_\l =-u_\l(s_\l)&=-\frac{r_{\l}}4 u_\l'(r_{\l})-\frac 14  \int_{r_\l}^{s_\l} r\left(|u_\l|^2+\l |u_\l|\right) dr.
		\end{align}
		Suppose by contradiction that $M_\l\to 0$, then $u_*\equiv 0$ and, passing to the limit in \eqref{eq:stima-norma-2} and using \eqref{stima-der-prima} we get 
		\[0= \frac{\bar\l}{2}, \]
		that gives a contradiction with $\bar \l>0$. Hence $M_\l\to M>0$, and writing 
		\eqref{eq:sol-int} with $r=s_\l$, passing to the limit we get
		\[0<M=\int_{s_1}^1 s^{-5}\int_{s_1}^s t^5 \left(\bar \l u_*+|u_*| u_*\right) dt ds\]
		where $s_1=\lim s_\l\geq 0$.  This shows that $u_*\neq 0$ as well. Then, by the definition of $\widehat u_\l$ the function $u_*$ has at most $m-1$ nodal zones in $(0,1)$. 
		\\
		Next we characterize the value of $M=\lim M_\l$. Assume that $r_2$ is the first zero of $u_*$ in the interval $(0,1]$. Then $u_*$ (or $-u_*$) is a positive radial solution to \eqref{i10} in $(0,r_2)$ and so, by the monotonicity result in \cite{GNN79} it satisfies $u_*'(0)=0$. 
		The convergence of $u_\l$ to $u_*$ in $C^1_{\loc}(\bar B\setminus\{0\})$ then implies that $s_\l \to 0$ as $\l\to \bar \l$.
		\\
		The boundedness of $u_\l$ in $(r_{\l},1)$ and the fact that $r_\l,s_\l\to 0$ as $\l\to \bar \l$ gives that 
		\[ \int_{r_\l}^{s_\l} r\left(|u_\l|^2+\l |u_\l|\right) dr\to 0\]
		 so that passing to the limit into \eqref{eq:stima-norma-2} and using \eqref{stima-der-prima} yields
		\[\lim _{\l \to \bar\l} \widehat u_\l(s_\l)=-\lim _{\l \to \bar\l} \frac{r_{\l}}4 u_\l'(r_{\l})=-\frac { \bar\l}2.
		\]
Afterwards the pointwise convergence of $\widehat u_\l$ to $u^*$ and \eqref{ff1}  implies that 
			\[\begin{split}
			|u^*(0)-\widehat u_\l(s_\l)| &\leq |u^*(0)-u^*(\bar r)|+|u^*(\bar r)-u_\l(\bar r)|+|u_\l(s_\l)-u_\l(\bar r)|\\
			&\leq C|\bar r|+|u^*(\bar r)-u_\l(\bar r)|+\max\limits_{r\in(s_\l,\bar r)}|u'_\l(r)|
			|s_\l-\bar r|\leq 2C|\bar r|+|u^*(\bar r)-u_\l(\bar r)|<\e
			\end{split}\]
		for any $\e>0$,	if $\l-\bar \l$ and  $\bar r$  are sufficiently small. This proves that $u^*(0)=-\frac {\bar \l}2$.

		Eventually we prove that $u^*$ has exactly $m-1$ nodal zones showing that $u^*=u^{m-1}$ as defined in \eqref{i10}. It follows passing to the limit in \eqref{eq:sol-int} or in \eqref{fra-nuova1}. Indeed if a nodal domain disappears then $u_*$ should satisfy $u_*(\bar r)=u_*'(\bar r)$ for some $\bar r\in [0,1]$, which is not possible.  
	\end{proof}

 As we said before, the main difficulty in the case $N=6$ is the computation of the rates of $r_\l$ and $\|u_\l\|_\infty$ with respect to $\l-\bar \l$. Our strategy goes as follows:
first, we build a radial   nodal solution to \eqref{1} using the Ljapunov-Schmidt procedure, next we deduce the asymptotics  of this solution
and finally  we prove the theorem using the uniqueness of the radial solution (as proved in Proposition \eqref{un}). \\

First of all, it is necessary to introduce some  assumptions which are crucial in our argument and whose  validity will be discussed later.
 
  Let  $\bar u^{m-1}$ be a radial  solution to 
\begin{equation}\label{20}
\begin{cases}
-\Delta \bar u^{m-1}=|\bar u^{m-1}|\bar u^{m-1}+ \bar \lambda \bar u^{m-1}\ \hbox{in}\ B,\\
\bar u^{m-1}(0)=\frac{\bar \lambda}{2}\\
\bar u^{m-1}\hbox{ has $m-1$ nodal zones }\\
\bar u^{m-1}=0\ \hbox{on}\ \partial B.
 \end{cases}
 \end{equation}
Observe that $\bar u^{m-1}=-u^{m-1}$ where $u^{m-1}$ is as defined in \eqref{i10}. We change this notations because it is easier for us to construct the solution to \eqref{1} using as limit function  $\bar u^{m-1}$ instead of $u^{m-1}$. 
 The solution we will find then will satisfies $u_\l(0)<0$ and we will recover our solution just multiplying by $-1$.

We also need to assume that $\bar u^{m-1}$ is non-degenerate, i.e.  
\begin{equation}\label{30}\boxed{
\begin{cases}
-\Delta v =(2|u^{m-1}| +\bar  \lambda)v \ \hbox{in}\  B\\
 v=0\ \hbox{on}\ \partial B.
 \end{cases}\Rightarrow\ v\equiv0}
 \end{equation}
If $\bar v_0$ is the solution of 
\begin{equation}\label{v0}
\begin{cases}
-\Delta \bar v_0-(2 |\bar u^{m-1}|+\bar \lambda )\bar v_0=\bar u^{m-1}\quad\mbox{in}\,\, \Om\\
\bar v_0=0\ \hbox{on}\ \partial B,
 \end{cases}
 \end{equation}
 whose existence is due to the non-degeneracy of $u^{m-1}$ (or of $\bar u^{m-1}$), we finally need to 
suppose  that
\begin{equation}\label{v00}\boxed{2\bar v_0(0) \not=1}\end{equation}
As before we have that $\bar v_0=-v_0$ where $v_0$ is as defined in \eqref{v0i}.

Next, we need to introduce    the well-known bubbles (see \cite{A, cgs, T})  
 \begin{equation}\label{Udxi} U_{\delta}(x):=\delta^{-2}U\(\frac{x }\delta\),\ \hbox{with}\ \delta>0,\  0 \in \R\ \hbox{and}\  U(y):=\frac{\alpha_6 }{\(1+|y|^2\)^2},\ \alpha_6:=24,\end{equation} 
which  are all the radial
positive solutions of the  Sobolev critical equation $$-\Delta U= U^2\ \mbox{in}\  \R^6.$$ We denote by $P U_{\delta}$ the projection   onto  $H^1_{0}( B),$ i.e. 
\begin{equation}\label{PU}
\begin{cases}
-\Delta P U_{\delta}=U^2_{\delta}\quad \mbox{in}\,\, \Om\\
 P U_{\delta}=0\quad \mbox{on}\,\, \partial\Om.
 \end{cases}
 \end{equation}
Finally we consider the  problem
\begin{equation}\label{p1}
\begin{cases}
-\Delta u=u|u| +(\bar \lambda+\e) u\ \hbox{in}\  B,\\
 u=0\ \hbox{on}\ \partial B,
 \end{cases}
 \end{equation}
where $\e$ is small enough (not necessarily positive).

Our  existence result reads as follows.
\begin{theorem}\label{main1} Assume  \eqref{30} and \eqref{v00}. There exists $\varepsilon_0>0$ such that
\begin{enumerate}
\item if $1-2\bar v_0(0) >0$   and  $\e\in(0,\varepsilon_0)$ , \\
 or
\item if $1-2\bar v_0(0) <0$  and $\e\in(- \varepsilon_0,0)$  
\end{enumerate}
then  there exists a radial nodal solution $u_\e$  to \eqref{p1} with $m$ nodal regions which blows-up  in the origin as $\e\to0$. More precisely
\begin{equation}\label{n11}
u_\e (x)=\bar u^{m-1}(x)+\e \bar v_0(x)-PU_{\delta_\e }(x)+\phi_\e(x)
\end{equation}

with as $\e\to0$ 
\begin{equation}
\label{scelta-d}
\delta_\e\e^{-1}\to d= \frac{8\sqrt{3}}{11}\frac{|1-2\bar v_0(0)|}{\bar \l ^\frac32}>0
\end{equation} 
and
$$\|\phi_\e\|_{H^1_{0}( B)}=\mathcal O\(\e^2|\ln|\e||^{\frac23}\).$$
\end{theorem}

\subsection{Proof of Theorem \ref{main1}}

\subsubsection{ Setting of the problem and the choice of the ansatz}
In what follows we denote by $$(u , v):=\int_\Om \nabla u\nabla v\, dx,\qquad \|u\|:=\(\int_\Om |\nabla u|^2\, dx\)^{\frac 12}$$
the inner product and its correspond norm in $H^1_{0}(\Om)$ while we denote by $$|u|_r:=\(\int_\Om |u|^r\, dx\)^{\frac 1 r}$$ the $L^r(\Om)$ standard norm. When $A\neq  B$ is any Lebesgue measurable set we specify the domain of integration by using $\|u\|_A, |u|_{r, A}$. \\ 
Let  $(-\Delta)^{-1}: L^{\frac 32}(\Om)\to H^1_{0}(\Om)$,
 be the operator defined as the unique solution of the equation $$-\Delta u =v \quad \mbox{in}\,\, \Om\qquad u=0\quad\mbox{on}\,\, \partial\Om.$$ By the Holder inequality it follows that 
$$\|(-\Delta)^{-1}(v)\|\le C |v|_{\frac 32}\qquad\forall\,\, v\in L^{\frac 32}(\Om)$$ for some positive constant $C$, which does not depend on $v.$ Hence we can rewrite problem \eqref{p1} as \begin{equation}\label{pb1r} u=(-\Delta)^{-1}[f(u)+(\bar \lambda+\e) u]\quad u\in H^1_{0,rad}(\Om)\end{equation} with $f(u)=|u|u$.\\

Next we remind the expansion of the projection of the bubble defined in \eqref{PU}. As before we denote by $G(x, y)$ the Green's function of the Laplace operator given by 
\begin{equation}\label{green}
G(x, y)=\frac{1}{4\sigma_6}\(\frac{1}{|x-y|^4}-H(x, y)\)\end{equation}
where $\sigma_6$ denotes the surface area of the unit sphere in $\R^6$ and $H$ is the regular part of the Green's function, namely for all $y\in\Om$, $H(x, y)$ satisfies $$\Delta H(x, y)=0\quad\mbox{in}\,\, \Om\qquad H(x, y)=\frac{1}{|x-y|^4}\,\, x\in\partial\Om.$$
It is known that the following expansion holds (see \cite{R})
\begin{equation}\label{varphiexp}PU_{\delta}(x)=U_{\delta}(x)-\alpha_6 \delta^2 H(x,  0 )+\mathcal O\(\delta^4\)\ \hbox{as}\ \delta\to0\end{equation}
uniformly in $ \bar B$.

Moreover we recall (see \cite{B}) that all the  solutions to the linear equation $$-\Delta\psi=2U_{\delta}\psi\quad\mbox{in}\,\, \R^6$$  in $H^1_{0,rad}(B)$, i.e. the subspace of radial functions in $H^1_0(B)$ are given by
$$\psi(x)=cZ_{\delta}(x), \ \hbox{with}\ Z_\delta(x)=\partial_\delta U_{\delta}(x)=2\alpha_6\delta \frac{|x |^2-\delta^2}{\(\delta^2+|x |^2\)^3}\ \hbox{and}\ c\in \R$$
 for $U_\delta$ and $\a_6$ as in \eqref{Udxi}.
Let $P Z_{\delta}$ be  the projection of $Z_{\delta}$ onto   $H^1_{0}(\Om)$, i.e.
 \begin{equation}\label{pbZ}-\Delta PZ_{\delta}= f'(U_{\de })Z_{\delta}\quad\mbox{in}\,\,  B\,\,\mbox{and}\quad PZ_{\delta}=0\quad\mbox{on}\,\,\, \partial B, \end{equation} elliptic estimates give $$PZ_{\delta}(x)=Z_{\delta}(x)-2\delta \alpha_6 H(x,  0 )+\mathcal O\(\delta^3\)\ \hbox{as}\ \delta\to0$$
 uniformly in $ \bar B$,  see \cite{R}. \\ 
We  look for a radial solution of \eqref{p1} of the form
\begin{equation}\label{sol}
u_\e(x)=\underbrace{\bar u^{m-1}(x)+\e \bar v_0(x)-PU_{\delta}(x)}_{:=W_{\delta}}+\phi_\e(x)
\end{equation}
where $\delta$ are chosen so that
\begin{equation}\label{sceltapar}
 \delta=|\e| d\ \hbox{with}\ d\in\(\sigma, \frac{1}{\sigma}\) \ \hbox{where $\sigma>0$ is small}
\end{equation} and $\phi_\e$ is a  radial remainder term, which is small as $\e\to0$, which belongs to the space $\mathcal K_{\delta}^\bot$ defined by
$$\mathcal K_{\delta}:= \{\phi\in H^1_{0,rad}( B)\,\,:\phi=cP Z_{\delta},\,\, c\in\R\}\ \hbox{and}\ \mathcal K_{\delta}^\bot:=\{\phi\in H^1_{0,rad}( B)\,\,:\,\,\, (\phi, PZ_{\delta})=0 \}.$$ 
Let us denote by $\Pi_{\delta}$ and $\Pi_{\delta}^\bot$ the projection of $H^1_{0,rad}(\Om)$ on $\mathcal K_{\delta}$ and $\mathcal K_{\delta}^\bot$ respectively.\\
Then solving problem \eqref{pb1r} is equivalent to solve the system
\begin{equation}\label{sist1}
\Pi_{\delta}^\bot\left\{u_\e(x)-(-\Delta)^{-1}\left[f(u_\e)+(\bar \lambda+\e) u_\e\right]\right\}=0,\end{equation}
\begin{equation}\label{sist2}
\Pi_{\delta}\left\{u_\e(x)-(-\Delta)^{-1}\left[f(u_\e)+ (\bar \lambda+\e) u_\e\right]\right\}=0.\end{equation}\vskip0.2cm
\subsubsection{The remainder term: solving  equation (\ref{sist1})}
The equation \eqref{sist1} can be written as 
\begin{equation}\label{sist1r}
\mathcal L_{\delta}(\phi_\e)+\mathcal R_{\delta}+\mathcal N_{\delta}(\phi_\e)=0
\end{equation}
where
\begin{equation}\label{L}
\mathcal L_{\delta}(\phi_\e)=\Pi_{\delta}^\bot\left\{\phi_\e(x)-(-\Delta)^{-1}\left[f'(W_{\delta})\phi_\e+\lambda \phi_\e\right]\right\}
\end{equation}
\begin{equation}\label{R}
\mathcal R_{\delta}=\Pi_{\delta}^\bot\left\{W_{\delta}(x)-(-\Delta)^{-1}\left[f(W_{\delta})+\lambda W_{\delta}\right]\right\}
\end{equation}
\begin{equation}\label{N}
\mathcal N_{\delta}(\phi_\e)=\Pi_{\delta}^\bot\left\{-(-\Delta)^{-1}\left[f(W_{\delta}+\phi_\e)-f(W_{\delta})-f'(W_{\delta})\phi_\e\right]\right\}
\end{equation}
where $\mathcal L_{\delta}$ is the linearized operator at the approximate solution, $\mathcal R_{\delta}$ is the error term and $\mathcal N_{\delta}$ is a quadratic term in $\phi_\e$  and, as before $f(u)=|u|u$.\\
In what follows we estimate the $H^1_{0,rad}( B)$ - norm of the error term $\mathcal R_{\delta}$ 
\begin{lemma}\label{error}For any $\sigma>0$ 
there exist  $c>0$ and $\varepsilon_0>0$ such that for any $d>0$ satisfying \eqref{sceltapar} and for any $\e\in(-\varepsilon_0,\varepsilon_0)$
$$\|\mathcal R_{\delta} \|\le c \e^{2}|\ln|\e||^{\frac 23}.$$\end{lemma}
\begin{proof}
First we remark that
$$\begin{aligned}
&-\Delta W_{\delta}-f(W_{\delta})-(\bar \lambda+\e)W_{\delta}=-\Delta \bar u^{m-1}-\e\Delta \bar v_0-f(U_{\de})-f(\bar u^{m-1}+\e \bar v_0- PU_{\delta})\\
&-\bar \lambda \bar u^{m-1}-\bar \lambda\e \bar v_0+(\bar \lambda+\e)PU_{\delta}-\e \bar u^{m-1}-\e^2\bar v_0\\
&=-f(\bar u^{m-1}+\e \bar v_0- PU_{\delta})-U_{\de}^2+f(\bar u^{m-1})+\e\underbrace{\(-\Delta \bar v_0 -\bar \lambda \bar v_0 - \bar u^{m-1}\)}_{=2|\bar u^{m-1}| \bar v_0}+(\bar \lambda+\e)PU_{\delta}-\e^2\bar v_0. 
\end{aligned}$$
By the continuity of $\Pi_{\delta}^\bot$ we get that
$$\begin{aligned} \|\mathcal R_{\delta}\|&\le c\left|-\Delta W_{\delta}-f(W_{\delta})-(\bar \lambda+\e) W_{\delta}\right|_{\frac 32}\\
&\le c \underbrace{\left|-f(\bar u^{m-1}+\e \bar v_0- PU_{\delta})-(PU_{\delta})^2+f(\bar u^{m-1})+2\e|\bar u^{m-1}|\bar v_0\right|_{\frac 32}}_{(I)}+c\underbrace{\left|(PU_{\delta})^2-U_{\delta}^2\right|_{\frac 32}}_{(II)}\\
&+\underbrace{(\bar \lambda+\e)\left|PU_{\delta}\right|_{\frac 32}}_{(III)}+\underbrace{\e^2\left|\bar v_0\right|_{\frac 32}}_{:=\mathcal O\(\e^2\)}
\end{aligned}$$
First of all, we point out that
$$(III)\le  c |U_{\delta}|_{\frac 32}\le c \de^2|\ln\de|^{\frac 23}$$
and by \eqref{varphiexp}
 $$\begin{aligned}
(II)&\le c\(\int_\Om\underbrace{|PU_{\delta}-U_{\delta}|^{\frac 32}}_{=O(\delta^2)}\underbrace{|PU_{\delta}+U_{\delta}|^{\frac 32}}_{\le cU_{\delta}}\)^{\frac 23}\le c \delta^2 \(\int_\Om |U_{\delta}|^{\frac 32}\, dx\)^{\frac 23} =\mathcal O\( \de^4|\ln\de|^{\frac 23}\).\end{aligned}$$
First let us estimate $(I)$ in $B( 0 , \sqrt\de)$ and $ B\setminus B( 0 , \sqrt\de)$:
$$\begin{aligned} (I)&\le c \(\int_{B( 0 , \sqrt\de)}\big|f(\bar u^{m-1}+\e \bar v_0-PU_{\delta})|+(PU_{\delta})^2\big|^{\frac 32}\)^{\frac 23}\\ &+c\underbrace{\(\int_{B( 0 , \sqrt\de)}\big|f(\bar u^{m-1})+\e f'(\bar u^{m-1}) \bar v_0|^{\frac 32}\, dx\)^{\frac 23}}_{=\mathcal O(\delta^2)}\\
&+ c\(\int_{ B\setminus B( 0 , \sqrt \de)}\big|f(\bar u^{m-1}+\e \bar v_0-PU_{\de })-f(\bar u^{m-1})-  f'(\bar u^{m-1})  (\e \bar v_0-PU_{\de })\big|^{\frac 32}\)^{\frac 23}\\ 
&+c\(\int_{ B\setminus B( 0 , \sqrt \de)}\big|(PU_{\delta})^2+f'(\bar u^{m-1})PU_{\de}\big|^{\frac 32}\, dx\)^{\frac 23}\\
&=\mathcal O\(\delta^2|\ln\de|^{\frac 23}\),\end{aligned}$$
since
 by mean value Theorem (here $\theta\in[0,1]$)  
$$\begin{aligned}&\int_{B( 0 , \sqrt\de)}\big|f(\bar u^{m-1}+\e \bar v_0-PU_{\delta})+(PU_{\delta})^2\big|^{\frac 32}
\\
 &=2\int_{B( 0 , \sqrt\de)}f'(\theta(\bar u^{m-1}+\e \bar v_0)-PU_{\delta})(\bar u^{m-1}+\e \bar v_0)|^{\frac 32}\, dx
\\ &
\le c\underbrace{\int_{B( 0 , \sqrt\de)}|PU_{\delta}|^{\frac32}\, dx}_{=\mathcal O(\delta^3|\log\delta|)}+c\underbrace{\int_{B( 0 , \sqrt\de)}|\bar u^{m-1}+\e \bar v_0|^{3}\, dx}_{=\mathcal O(\delta^3)}
\end{aligned},$$
and by the inequality
\begin{equation}\label{a1}
 \big|f(a+b)-f(a)-f'(a)b\big|\le 7 b^2\ \hbox{for any}\ a,b\in \mathbb R
 \end{equation}
$$\begin{aligned}&\int_{ B\setminus B( 0 , \sqrt \de)}\Big|
f(\bar u^{m-1}+\e \bar v_0-PU_{\de})-f(\bar u^{m-1})- f'(\bar u^{m-1}) (\e \bar v_0-PU_{\de})\Big|^{\frac 32}\\
 &\le c\int_{ B\setminus B( 0 , \sqrt \de)}|\e \bar v_0-PU_{\de}|^{3}dx\\
 &\le c\underbrace{\int_{ B\setminus B( 0 , \sqrt \de)}|\e \bar v_0|^{3}dx}_{=\mathcal O(\e^3)}+c\underbrace{\int_{ B\setminus B( 0 , \sqrt \de)}| U_{\de }|^{3}dx}_{=\mathcal O(\delta^3)}\Rightarrow\\
 &\left(\int_{ B\setminus B( 0 , \sqrt \de)}\Big|
f(\bar u^{m-1}+\e \bar v_0-PU_{\de})-f(\bar u^{m-1})- f'(\bar u^{m-1}) (\e \bar v_0-PU_{\de})\Big|^{\frac 32}\right)
 ^\frac23=O(\e^2)
\end{aligned}$$
 and 
 $$\int_{ B\setminus B( 0 , \sqrt \de)}\Big|
 f(PU_{\de})+f'(\bar u^{m-1})PU_{\de}\Big|^{\frac 32}\le c\underbrace{\int_{ B\setminus B( 0 , \sqrt \de)}| U_{\de} |^3\, dx}_{=\mathcal O(\delta^3 )}+\underbrace{\int_{ B\setminus B( 0 , \sqrt \de)}| U_{\de }|^{\frac 32}\, dx}_{=\mathcal O(\delta^3|\log\delta|)}.$$
which ends the proof.
 \end{proof}
Next we state the invertibility of the linear operator $\mathcal L_{\delta}:\mathcal K_{\de}^\bot \to \mathcal K_{\de}^\bot$ defined in \eqref{L} and provide a uniform estimate of the norm of $\mathcal L_{\delta}^{-1}$ (see  for example \cite{va}, Lemma 2.4 or \cite{rv}, Lemma 4.2).
 
\begin{lemma}\label{inv}
For any $\sigma>0$ there exist $c>0$ and $\varepsilon_0>0$ such that for any $d>0$ satisfying \eqref{sceltapar} and for any $\e\in(-\varepsilon_0,\varepsilon_0)$ 
$$\|\mathcal L_{\delta}(\phi) \|\ge c \|\phi\|\ \hbox{for any}\ \phi\in \mathcal K_{\delta}^\bot.$$ Moreover, $\mathcal L_{\delta}$ is invertible and $\|\mathcal L_{\delta}^{-1}\|\le \frac 1 c.$\end{lemma}
We are in position now to find a solution of the equation \ref{sist1} whose proof relies on a standard contraction mapping argument (see for example \cite{mupi}, Proposition 1.8 and \cite{mipive}, Proposition 2.1)
\begin{proposition}\label{solphi}
For any $\sigma>0$ there exist $c>0$ and $\varepsilon_0>0$ such that for any $d>0$ satisfying \eqref{sceltapar} and for any $\e\in(-\varepsilon_0,\varepsilon_0)$ 
 there exists a unique $\phi_\e=\phi_\e(d)\in \mathcal K_{\delta}^\bot$ solution to \eqref{sist1} which is continuously differentiable with respect to $d$ and such that
\begin{equation}\label{stimaphi}
\|\phi_\e\|\le c \e^2|\ln|\e||^{\frac 23}. 
\end{equation} \end{proposition}

\subsubsection{The reduced problem: solving equation (\ref{sist2})}
To solve  equation \eqref{sist2}, we shall  find the parameter $\delta$    as in \eqref{sceltapar},  i.e. $d>0$,  so that \eqref{sist2} is satisfied. \\ It is well known that this problem has a variational structure, in the sense that solutions of \eqref{sist2} reduces to find critical points to some given explicit  function defined on $\R$. Indeed, let $J_\e: H^1_{0,rad}(\Om)\to \mathbb R$ defined by 
$$J_\e(u):=\frac 12 \int_\Om |\nabla u|^2\, dx-\frac {\bar \lambda+\e} 2 \int_\Om u^2\, dx-\frac 13\int_\Om |u|^3\, dx$$ 
and let $\tilde J_\e: (0,+\infty)\to \mathbb R$ be the reduced energy which is defined by $$\tilde J_\e(d)=J_\e(W_{\delta}+\phi_\e)$$
where $W_{\delta}$ is as defined in \eqref{sol} and $\phi_\e=\phi_\e(d)$ is the function found in Proposition \ref{solphi}.
\begin{proposition}\label{cruc0} For any $\sigma>0$ there exists $\varepsilon_0>0$ such that for any $\e\in(-\varepsilon_0,\varepsilon_0)$ 
\begin{equation}\label{cruc1}
\tilde J_\e(d)= \mathfrak c_0(\e)+|\e|^3\underbrace{\left\{ \mathtt{sgn}(\e)\left[1-2v_0(0)\right] d^2\mathfrak a_1-d^3\mathfrak a_2 \right	\}}_{=:\Upsilon(d)}+o\(|\e|^3\)
\end{equation}
uniformly with respect to $d$  in compact sets of $(0,+\infty)$, where  
  $\mathfrak c_0(\e)$ only depends on $\e$ and the $\mathfrak a_i$'s are positive constants.
Moreover,  if $d$  is a critical point of $\tilde J_\e$, then $W_{\delta}+\phi_\e$ is a solution of \eqref{p1}.\\
\end{proposition}

\begin{proof}
It is quite standard to prove that  if $d$ satisfies \eqref{sceltapar} and is a critical point of $\tilde J_\e$, then $W_{\delta}+\phi_\e$ is a solution of \eqref{p1}
(see for example \cite{mipive}, Proposition 2.2). Moreover, it is not difficult to check that $\tilde J_\e(d )=J_\e(W_{\delta})+o\(|\e|^3\) $ uniformly with respect to $d$  in compact sets of $(0,+\infty)$
  (see for example \cite{mipive}, Proposition 2.2). \\
Let us estimate the main term of the reduced energy, i.e. 
$$\begin{aligned} &J_\e(\bar u^{m-1}+\e \bar v_0-PU_{\delta})\\ &= \frac 1 2 \intO |\nabla(\bar u^{m-1}+\e \bar v_0- PU_{\delta})|^2-\frac{\bar \lambda+\e}{2}\intO(\bar u^{m-1}+\e \bar v_0- PU_{\delta})^2
-\frac 13\intO |\bar u^{m-1}+\e \bar v_0- PU_{\delta}|^3	\\
&=\frac 1 2 \intO |\nabla(\bar u^{m-1}+\e \bar v_0)|^2+\frac 12 \intO|\nabla PU_{\delta}|^2-\frac{\bar \lambda+\e}{2}\intO(\bar u^{m-1}+\e \bar v_0)^2-\frac{\bar \lambda+\e}{2}\intO(PU_{\delta})^2\\  &-\underbrace{\(\intO \nabla \bar u^{m-1}\nabla PU_{\delta}-\bar \lambda\intO \bar u^{m-1} PU_{\delta}\)}_{=\intO |\bar u^{m-1}|\bar u^{m-1} PU_{\delta}}-\e\underbrace{ \(\intO \nabla \bar v_0\nabla PU_{\delta}-\bar \lambda  \intO \bar v_0 PU_{\delta}- \intO \bar u^{m-1} PU_{\delta}\)}_{= \intO 2|\bar u^{m-1}| \bar v_0  PU_{\delta} }   \\
&+ \e^2\intO  \bar v_0PU_{\delta}-\frac 13\intO |\bar u^{m-1}+\e \bar v_0- PU_{\delta}|^3\\
&  =\underbrace{\frac 1 2 \intO |\nabla(\bar u^{m-1}+\e \bar v_0)|^2-\frac{\bar \lambda+\e}{2}\intO(\bar u^{m-1}+\e \bar v_0)^2-\frac13 \intO |\bar u^{m-1}+\e \bar v_0 |^3}_{=:I_1} \\
&+ \underbrace{\frac 12 \intO|\nabla PU_{\delta}|^2-\frac 13\intO PU_{\delta}^3}_{=:I_2}\underbrace{-\frac{\bar \lambda}{2}\intO  PU_{\delta}^2+\intO \bar u^{m-1}  PU_{\delta}^2}_{=:I_3}
\underbrace{-\frac\e 2\intO PU_{\delta}^2+\e\int \bar v_0 PU_{\delta}^2}_{=:I_4}\\ 
&\underbrace{ -\frac 13\intO\(|\bar u^{m-1}+\e \bar v_0- PU_{\delta}|^3-|\bar u^{m-1}+\e \bar v_0|^3- PU_{\delta}^3+3(\bar u^{m-1}  +\e \bar v_0)PU_{\delta}^2+3|\bar u^{m-1} +\e \bar u^{m-1}|(\bar u^{m-1}+\e \bar v_0) PU_{\delta}\)}_{=:I_5}\\&
\underbrace{ + \intO\Big[|\bar u^{m-1} +\e \bar v_0|(\bar u^{m-1}+\e \bar v_0) -(|\bar u^{m-1}|\bar u^{m-1} +2\e|\bar u^{m-1}|  \bar v_0)\Big] PU_{\delta} }_{=:I_6}+ \underbrace{\e^2\intO \bar v_0PU_{\delta}}_{=:I_7}\end{aligned}$$
It is clear that
$$I_7=\mathcal O\(\e^2\intO{\delta^2\over |x |^4}dx\right)=\mathcal O\(\e^2\delta^2\)=\mathcal O\(\e^4\).$$
To estimate $I_6$ by \eqref{a1} it follows that
$$I_6=O\left(\e^2\intO PU_{\delta}
\right)=O\(\e^2\delta^2\)=\mathcal O\(\e^4\).$$
Now, $I_1$  does not depend on $d$ and it will be included in the constant $\mathfrak c_0(\e)$ in \eqref{cruc1}.
By \eqref{varphiexp}
\begin{equation*}
\begin{aligned}
&I_2= \frac 12 \intO U_{\delta}^2 PU_\d
-\frac 13\intO PU_{\delta}^3\\
&=
 \frac 12 \intO U_{\delta}^2 \big(U_{\delta}(x)-\alpha_6 \delta^2 H(x,  0 )+\mathcal O\(\delta^4\)\big)
-\frac 13\intO\big(U_{\delta}(x)-\alpha_6 \delta^2 H(x,  0 )+\mathcal O\(\delta^4\)\big)^3\\
&=
\frac1 6
\int\limits_{\R^6}U_{\delta}^3+\mathcal O\left(\delta^2\intO U_{\delta}^2\right)+O\(\delta^4\)=\frac1 6\int\limits_{\R^6}U^3+O\(\delta^4\).
 \end{aligned}
\end{equation*}
Now, 
 setting $\varphi_{\delta}:=PU_{\delta}-U_{\delta}=\mathcal O(\delta^2),$ by \eqref{varphiexp} and \eqref{sceltapar}
\begin{equation*}
\begin{aligned}
I_3&=\intO \( \bar u^{m-1}(x)-\frac{\bar \lambda}2\)(U_{\delta}+\varphi_{\delta})^2\\
& =\intO \( \bar u^{m-1}(x)-\bar u^{m-1}(0)\)U_{\delta}^2+\mathcal O(\delta^4)\\
&=\intO\left[\frac12\langle D^2 \bar u^{m-1}(0)x,x\rangle+\mathcal O(|x|^3)\right]\alpha_6^2\frac{\delta^4}{(\delta^2+|x |^2)^4}dx +\mathcal O(\delta^4)\\
&=\alpha_6^2\intO\frac12 \langle D^2 \bar u^{m-1}(0)x,x\rangle \frac{\delta^4}{(\delta^2+|x |^2)^4}dx +\mathcal O(\delta^4)\\
&=\alpha_6^2\delta^2\int\limits_{ B- 0 \over\delta}\frac12 \langle D^2 \bar u^{m-1}(0)\delta y ,\delta y\rangle \frac{1}{(1+|y|^2)^4}dy +\mathcal O(\delta^4)\\
&=\mathcal O(\delta^4|\ln\delta|=\mathcal O(\e^4|\ln|\e||)
 \end{aligned}
\end{equation*}
and analogously 
\begin{equation*}
\begin{aligned}
&I_4=\e\intO\(\bar v_0(x)-\frac12\) PU_{\delta}^2=\e\left[\alpha_6^2\delta^2\(\int\limits_{\R^6}\frac{1}{(1+|y|^2)^4}dy\) \(\bar v_0(0)-\frac12\)  +o(1)\right]\\
&=\e^3d^2\left[\alpha_6^2\(\int\limits_{\R^6}\frac{1}{(1+|y|^2)^4}dy\) \(\bar v_0(0)-\frac12\)  +o(1)\right]
 \end{aligned}
\end{equation*}

Finally, we have to estimate $I_5.$

We point out that
$$|\bar u^{m-1}+\e \bar v_0- PU_{\delta}|^3-|\bar u^{m-1}+\e \bar v_0|^3- PU_{\delta}^3+3(\bar u^{m-1}  +\e \bar v_0)PU_{\delta}^2+3|\bar u^{m-1} +\e \bar u^{m-1}|(\bar u^{m-1}+\e \bar v_0) PU_{\delta}=0\ \hbox{if}\ \bar u^{m-1}+\e \bar v_0\le0$$
and so
$$\begin{aligned} I_5  &=-\frac 13\int_{\{\bar u^{m-1}+\e \bar v_0\ge0\}}\(|\bar u^{m-1}+\e \bar v_0- PU_{\delta}|^3-(\bar u^{m-1}+\e \bar v_0)^3- PU_{\delta}^3+3(\bar u^{m-1}  +\e \bar v_0)PU_{\delta}^2+
3 (\bar u^{m-1}+\e \bar v_0)^2 PU_{\delta}\)dx \\
&=-\frac 13\int_{\{\bar u^{m-1}+\e \bar v_0\ge PU_{\delta}\}}\left(-2PU_{\delta}^3+6(\bar u^{m-1}  +\e \bar v_0) PU_{\delta}^2\right)\\
&-\frac 13\int_{\{0< \bar u^{m-1}+\e \bar v_0<PU_{\delta}\}}\left(-2(\bar u^{m-1}+\e \bar v_0)^3+6(\bar u^{m-1}  +\e \bar v_0)^2 PU_{\delta} \right).\end{aligned}$$
 
First of all we claim that for any $\sigma>0$ there exists $\varepsilon_0>0$ such that for any $\e\in(-\varepsilon_0,\varepsilon_0)$ and $d>0$ satisfying \eqref{sceltapar}
\begin{equation}\label{claimlevel}
B\( 0 , R^1_\delta\sqrt\delta\)\subset\{x\in B\,:\, 0< \bar u^{m-1}(x)+\e \bar v_0(x)< P U_{\delta}(x)\} \cap B\( 0 ,\delta^\frac14\)\subset B\( 0 , R^2_\delta\sqrt\delta\)
\end{equation}
where
\begin{equation}\label{ro}  R^1_\delta,R^2_\delta=R_0+o(1)\ \hbox{with}\  R_0:=\(\frac{\alpha_6}{\bar u^{m-1}(0)}\)^{\frac 14}. \end{equation}
We remind that $\delta=\mathcal O(\e)$ and also that   $P U_{\delta}(x)=\alpha_6{\delta^2\over(\delta^2+|x |^2)^2}+\mathcal O(\e^2)$ uniformly in $ B.$
If $|x |<R^1_\delta\sqrt\delta$ is small enough then by mean value theorem $\bar u^{m-1}(x)+\e \bar v_0(x)=\bar u^{m-1}(0)+\mathcal O_1(\e) $ and
$$ \begin{aligned}\bar u^{m-1}(x)+\e \bar v_0(x)< P U_{\delta}(x)\ &\Leftrightarrow\ \frac{\bar u^{m-1}(0)}{\alpha_6}+\mathcal O_1(\e)< {\delta^2\over(\delta^2+|x |^2)^2}\\
& \Leftrightarrow\ 
|x |\le \sqrt\delta \underbrace{\({1\over  \(\frac{\bar u^{m-1}(0)}{\alpha_6}+\mathcal O_1(\e)\)^{\frac12}}-\delta\)^{1\over2}}_{R^1_\delta}\end{aligned}$$
and the first inclusion in \eqref{claimlevel} together with \eqref{ro} follow.
On the other hand, again by mean value theorem 
we have  $\bar u^{m-1}(x)+\e \bar v_0(x)=\bar u^{m-1}(0)+\mathcal O_2(\sqrt \delta)$ for any $x\in B( 0 ,\delta^\frac14)$  
and arguing as above we get the second inclusion in \eqref{claimlevel} and \eqref{ro}. \\

It is useful to point out that by \eqref{claimlevel} we immediately get
\begin{equation}\label{claimlevel+}
B^c\( 0 , R^1_\delta\sqrt\delta\) \supset 
\{x\in B\,:\,   \bar u^{m-1}(x)+\e \bar v_0(x)\ge P U_{\delta}(x)\}\cup B^c\( 0 ,\delta^\frac14\)
\supset B^c\( 0 , R^2_\delta\sqrt\delta\)
\end{equation}
 Now by \eqref{claimlevel} and \eqref{claimlevel+} we deduce
$$\begin{aligned} I_5&=-\frac 13\int_{\{\bar u^{m-1}+\e \bar v_0\ge PU_{\delta}\}}\left(-2PU_{\delta}^3+6(\bar u^{m-1}  +\e \bar v_0) PU_{\delta}^2\right)\\
&-\frac 13\int_{\{0<\bar u^{m-1}+\e \bar v_0<PU_{\delta}\}}\left(-2(\bar u^{m-1}+\e \bar v_0)^3+6(\bar u^{m-1}  +\e \bar v_0)^2 PU_{\delta} \right)\\
&=-\frac 13\int_{\{\bar u^{m-1}+\e \bar v_0\ge PU_{\delta}\}\cup B^c\( 0 ,\delta^\frac14\)}\left(-2PU_{\delta}^3+6(\bar u^{m-1}  +\e \bar v_0) PU_{\delta}^2\right)\\
&+\frac 13\int_{B^c\( 0 ,\delta^\frac14\)\setminus\left[ \{\bar u^{m-1}+\e \bar v_0\ge PU_{\delta}\}\cap B^c\( 0 ,\delta^\frac14\)\right] }\left(-2PU_{\delta}^3+6(\bar u^{m-1}  +\e \bar v_0) PU_{\delta}^2\right)\\
&-\frac 13\int_{\{0<\bar u^{m-1}+\e \bar v_0<PU_{\delta}\}\cap B\( 0 ,\delta^\frac14\)}\left(-2(\bar u^{m-1}+\e \bar v_0)^3+6(\bar u^{m-1}  +\e \bar v_0)^2 PU_{\delta} \right)\\
& -\frac 13\int_{\{0<\bar u^{m-1}+\e \bar v_0<PU_{\delta}\}\cap B^c\( 0 ,\delta^\frac14\)}\left(-2(\bar u^{m-1}+\e \bar v_0)^3+6(\bar u^{m-1}  +\e \bar v_0)^2 PU_{\delta} \right)\\
  &=-\frac 13\int_{\{\bar u^{m-1}+\e \bar v_0\ge PU_{\delta}\}\cup B^c\( 0 ,\delta^\frac14\)  }\left(-2PU_{\delta}^3+6(\bar u^{m-1}  +\e \bar v_0) PU_{\delta}^2\right)\\ 
  &-\frac 13\int_{\{0<\bar u^{m-1}+\e \bar v_0<PU_{\delta}\} \cap B\( 0 ,\delta^\frac14\)}\left(-2(\bar u^{m-1}+\e \bar v_0)^3+6(\bar u^{m-1}  +\e \bar v_0)^2 PU_{\delta} \right)+o(\delta^3),
\end{aligned}$$
because
$$\begin{aligned} \int_{B^c\( 0 ,\delta^\frac14\)\setminus\{\bar u^{m-1}+\e \bar v_0\ge PU_{\delta}\}\cap B^c\( 0 ,\delta^\frac14\)}\left(-2PU_{\delta}^3+6(\bar u^{m-1}+\e \bar v_0) PU_{\delta}^2  \right)&=\mathcal O\(
\int_{  B^c\( 0 ,\delta^\frac14\)}\left(U_{\delta}^3+ U_{\delta}^2  \right)\)\\
&=\mathcal O\(\delta^\frac72\),\end{aligned}$$
and, since $0<\bar u^{m-1}+\e \bar v_0<PU_{\delta}$
\begin{equation}\label{a3}
\begin{split}
&\int_{\{0<\bar u^{m-1}+\e \bar v_0<PU_{\delta}\}\cap B^c\( 0 ,\delta^\frac14\)}\left(-2(\bar u^{m-1}+\e \bar v_0)^3+6(\bar u^{m-1}  +\e \bar v_0)^2 PU_{\delta} \right)=\mathcal O\( \int_{  B^c\( 0 ,\delta^\frac14\)}U_{\delta}^3
\)=\mathcal O\(\delta^{\frac 92}\)=\\
&o(\delta^3)
\end{split}
\end{equation}
We estimate the last two  terms in the expansion of $I_5.$ 
Recalling that $PU_{\delta}=U_{\delta}+O(\d^2)$ uniformly in $B$ we have, for every $\bar r>0$, 
\[\begin{split}
&\int_{\{\bar u^{m-1}+\e \bar v_0\ge PU_{\delta}\}\cup B^c\( 0 ,\delta^\frac14\)  }\left(-2PU_{\delta}^3+6(\bar u^{m-1}  +\e \bar v_0) PU_{\delta}^2\right)=\\
&=\int_{ \{x\in B_{\bar r} :\bar u^{m-1}+\e \bar v_0\ge PU_{\delta}\}\cup \{\delta^\frac14\le | x|<\bar r\}}\left(-2PU_{\delta}^3+6(\bar u^{m-1}  +\e \bar v_0) PU_{\delta}^2\right)+O(\delta^4)
%&=\int_{\left( \{\bar u^{m-1}+\e \bar v_0\ge PU_{\delta}\}\cup B^c\( 0 ,\delta^\frac14\) \right)\cap B_{\bar r} }\left(-2PU_{\delta}^3+6\bar u^{m-1}  PU_{\delta}^2\right)+O(\delta^4){\M(\hbox{ forse si pu\'o tenere }\e\bar v_0)}
\end{split}
\]
Next, we let $\xi_1$ be the first zero of $\bar u^{m-1}$ in $(0,1]$ so that $\bar u^{m-1}>0$ in $(0,\xi_1)$. The function $\bar u^{m-1}$ is radially decreasing in $(0,\xi_1)$ (by \cite{GNN79}) and so we have $\bar u^{m-1}>\bar u^{m-1}(\frac {\xi_1}2)=c_1>0$ in $B_{\frac {\xi_1}2}$.
Next we claim that 
\begin{equation}\label{ff3}
-2PU_{\delta}+6(\bar u^{m-1} +\e \bar v_0)>0\hbox { in }
\left\{x\in B_{\frac{\xi_1}2} : \bar u^{m-1}+\e \bar v_0\ge PU_{\delta}\right\}\cup\  \left\{\delta^\frac14\le|x|\le\frac{|\xi_1|}2\right\}
\end{equation}
if $\delta$ is small enough. This is easily true in the set $\{\bar u^{m-1}+\e \bar v_0\ge PU_{\delta}\}$ since $6(\bar u^{m-1}+\e \bar v_0)-2PU_{\delta}\geq 4PU_{\delta}\geq 0$.
 %\bar u^{m-1}+5 \bar u^{m-1}-2PU_{\delta}\geq \bar u^{m-1}(\frac {\xi_1}2)+3PU_{\delta}-5\e \bar v_0\geq \frac{c_1}2+3PU_{\delta}\geq 0$ if $\e$ is small enough. 
 In the set %$B^c\( 0 ,\delta^\frac14\)\cap B_{\frac {\xi_1}2}$ 
 $\{\delta^\frac14\le|x|\le\frac{|\xi_1|}2\}$
 instead we have that $U_\delta\to0$, and since  $\bar u^{m-1}<c_1$ we get that \eqref{ff3} holds if $\delta$ and $\e$ are small enough.
Moreover, by \eqref{claimlevel+}
\begin{equation*}
\left\{R^2_\delta\sqrt\delta<|x|<\frac{\xi_1}2\right\}\subset
\left\{x\in B_{\frac{\xi_1}2} :\bar u^{m-1}+\e \bar v_0\ge PU_{\delta}\right\}\cup \
\left\{\delta^\frac14\le | x|<\frac {\xi_1}2\right\}
\subset\left\{R^1_\delta\sqrt\delta<|x|<\frac{\xi_1}2\right\}\
\end{equation*}
so that the \eqref{ff3} gives 
\[\begin{split}
&\int_{R^2_\delta\sqrt\delta<|x|<\frac{\xi_1}2
}\left(-2PU_{\delta}^3+6(\bar u^{m-1}  +\e \bar v_0) PU_{\delta}^2\right)\\
& \ \ \ \leq \int_{ \{x\in B_{\frac{\xi_1}2} :\bar u^{m-1}+\e \bar v_0\ge PU_{\delta}\}\cup \{\delta^\frac14\le | x|<\frac{\xi_1}2\}}\left(-2PU_{\delta}^3+6(\bar u^{m-1}  +\e \bar v_0) PU_{\delta}^2\right)\\
& \ \ \ \leq 
\int_{R^1_\delta\sqrt\delta<|x|<\frac{\xi_1}2
}\left(-2PU_{\delta}^3+6(\bar u^{m-1}  +\e \bar v_0) PU_{\delta}^2\right)
\end{split}
\]

Now if $R_\delta$ denotes either $R^1_\delta$ or $R^2_\delta$ we get
$$\begin{aligned}
& \int_{R_\de\sqrt\de<|x |<\frac{\xi_1}2}\left(-2PU_{\delta}^3+6(\bar u^{m-1}+\e \bar v_0) PU_{\delta}^2  \right)\\ 
&=-2\int_{R_\de\sqrt\de<|x |<\frac{\xi_1}2}U_{\delta}^3+6\int_{R_\de\sqrt\de<|x |<\frac{\xi_1}2}\bar u^{m-1} U_{\delta}^2+\mathcal O\(\d^4\)\\
&  =-2\int_{\frac{R_\de}{\sqrt\de}<|y |<\frac{\xi_1}{2\delta}}\frac{\alpha_6^3}{(1+|y|^2)^6}+6\de^2\int_{\frac{R_\de}{\sqrt\de}<|y |<\frac{\xi_1}{2\delta}}\bar u^{m-1}(\de y+ 0 )\frac{\alpha_6^2}{(1+|y|^2)^4}+\mathcal O\(\de^4\)\\
& =-2\sigma_6\alpha_6^3\int_{\frac{R_\de}{\sqrt\de}}^{\frac{\xi_1}{2\delta}}\frac{r^5}{(1+r^2)^6}+6\de^2\sigma_6\alpha_6^2 \bar u^{m-1}(0)\int_{\frac{R_\de}{\sqrt\de}}^{\frac{\xi_1}{2\delta}}\frac{r^5}{(1+r^2)^4}+\mathcal O\(\de^4\int_{\frac{R_\de}{\sqrt\de}}^{\frac{\xi_1}{2\delta}}\frac{r^7}{(1+r^2)^4}\)+\mathcal O\(\de^4\)\\
&= -\frac 13\sigma_6\alpha_6^3 R_\de^{-6}\de^3+3\de^3\sigma_6\alpha_6^2R_\de^{-2}\bar u^{m-1}(0)+\mathcal O\(\de^4|\log\de|\)\\
&= -\frac 13\sigma_6\alpha_6^3 R_0^{-6}\de^3+3\de^3\sigma_6\alpha_6^2R_0^{-2}\bar u^{m-1}(0)+o\(\de^{3}\), \ \hbox{because of \eqref{ro}}
\end{aligned}$$
and by comparison
\begin{equation}\label{fuoripalla}
\begin{aligned}
&\int_{ \{x\in B_{\frac{\xi_1}2} :\bar u^{m-1}+\e \bar v_0\ge PU_{\delta}\}\cup \{\delta^\frac14\le | x|<\frac{\xi_1}2\}}\left(-2PU_{\delta}^3+6(\bar u^{m-1}  +\e \bar v_0) PU_{\delta}^2\right)\\
&= -\frac 13\sigma_6\alpha_6^3 (R_0)^{-6}\de^3+3\de^3\sigma_6\alpha_6^2(R_0)^{-2}\bar u^{m-1}(0)+o\(\de^{3}\).
\end{aligned}\end{equation}
Finally let us estimate the last term in $I_5.$ First we note that
by \eqref{claimlevel} we immediately get
$$\begin{aligned}
6\int_{|x |<R_\delta^1\sqrt\de}\!\!\!\!\!\!\!\!\!\!
(\bar u^{m-1}+\e\bar v_0)^2 PU_{\delta} &\le 6\int_{\{0< \bar u^{m-1}+\e \bar v_0<PU_{\delta}\}\cap  B\( 0 ,\delta^\frac14\)}
\!\!\!\!\!\!\!\!\!\!\!\!\!\!\!\!\!\!\!\!\!\!\!\!\!
(\bar u^{m-1}+\e\bar v_0)^2 PU_{\delta}\\ 
&\le 6\int_{|x |<R_\de^2\sqrt\de}(\bar u^{m-1}+\e\bar v_0)^2 PU_{\delta}\end{aligned}$$
and, when $\d$ is small enough
$$\begin{aligned}-2\int_{|x |<R_\delta^2\sqrt\de}\!\!\!\!\!\!\!\!\!\!
(\bar u^{m-1}+\e\bar v_0)^3&\le -2\int_{\{0< \bar u^{m-1}+\e \bar v_0<PU_{\delta}\}\cap  B\( 0 ,\delta^\frac14\)}
\!\!\!\!\!\!\!\!\!\!\!\!\!\!\!\!\!\!\!\!\!\!\!\!\!
(\bar u^{m-1}+\e\bar v_0)^3\\ &\le -2\int_{|x |<R_\de^1\sqrt\de}(\bar u^{m-1}+\e\bar v_0)^3\end{aligned}$$
since $\bar u^{m-1}+\e\bar v_0>0$ for $|x |$ sufficiently small. 

If $R_\delta$ denotes either $R^1_\delta$ or $R^2_\delta$ we get 
$$\begin{aligned}
& -2\int_{|x |<R_\de\sqrt\de}\ (\bar u^{m-1}+\e\bar v_0)^3 =-2\de^6(1+\mathcal O(\de))\int_{|y|<\frac{R_\de}{\sqrt\de}}(\bar u^{m-1}(\de y))^3+\mathcal O\(\de^5\)\\
& =\(-2(\bar u^{m-1}(0))^3+\mathcal O\(\sqrt\de\)\)\de^6\sigma_6\int_0^{\frac{R_\de}{\sqrt \de}}r^5
+\mathcal O\(\de^5\)\\
& =-2\de^3(\bar u^{m-1}(0))^3\sigma_6R_\de^6+\mathcal O\(\de^{\frac 72}\)\\
& =-2\de^3(\bar u^{m-1}(0))^3\sigma_6R_0^6+o\(\de^3\), \ \hbox{because of \eqref{ro}}
\end{aligned}$$
and 
$$\begin{aligned}
& 6\int_{|x |<R_\de\sqrt\de}
(\bar u^{m-1}+\e\bar v_0)^2 PU_{\delta} =6\de^4(1+\mathcal O(\de))\int_{|y|<\frac{R_\de}{\sqrt\de}}(\bar u^{m-1}(\de y ))^2\frac{\alpha_6}{(1+|y|^2)^2}+\mathcal O\(\de^5\)\\
& =6\alpha_6\((\bar u^{m-1}(0))^2+\mathcal O\(\sqrt\de\)\)\de^4\sigma_6\int_0^{\frac{R_\de}{\sqrt\de}}\frac{r^5}{(1+r^2)^2}+\mathcal O\(\de^5\)\\
& =3\alpha_6\de^3(\bar u^{m-1}(0))^2\sigma_6R_\de^2+\mathcal O\(\de^{\frac 72}\)\\
& =3\alpha_6\de^3(\bar u^{m-1}(0))^2\sigma_6R_0^2+o\(\de^3\), \ \hbox{because of \eqref{ro}}
\end{aligned}$$
and by comparison
\begin{equation}\label{dentropalla}
\begin{split}
&\!\!\!\!\!\!\!\!\!\!
\int_{\{\bar u^{m-1}+\e \bar v_0<PU_{\delta}\}\cap B\( 0 ,\delta^\frac14\)}\left(-2(\bar u^{m-1}+\e\bar v_0)^3+6(\bar u^{m-1}+\e\bar v_0)^2 PU_{\delta} \right)\\
& \ \ \ \ \ \ =-2\de^3(\bar u^{m-1}(0))^3\sigma_6R_0^6+3\alpha_6\de^3(\bar u^{m-1}(0))^2\sigma_6R_0^2+o\(\de^{3}\)
\end{split}\end{equation}

Finally, by \eqref{fuoripalla} and \eqref{dentropalla}
 $$I_5=|\e|^3d^3\(-\frac{11}9\sigma_6\alpha_6^{\frac 32}(\bar u^{m-1}(0))^{\frac 32}+o(1)\).$$
%$$\begin{aligned}
%&=\frac 2 3 \int_{|x |>R_0\sqrt\delta}U_{\delta}^3-2\int_{|x |>R_0\sqrt\delta}u_0(x)U_{\delta}^2+\e^2 \int_{|x |>R_0\sqrt\delta}v_0^2 PU_{\delta}-2\e \int_{|x |>R_0\sqrt\delta} v_0 U_{\delta}^2\\
%&+\frac 23 \int_{|x |<R_0\sqrt\delta}u_0^3(x)-2\int_{|x |<R_0\sqrt\delta}u_0^2(x) U_{\delta}+2\e \int_{|x |<R_0\sqrt\delta}u_0^2 v_0+2\e^2 \int_{|x |<R_0\sqrt\delta}u_0 v_0^2\\
%&+\frac 23\e^3 \int_{|x |<R_0\sqrt\delta} v_0^3 -\e^2 \int_{|x |<R_0\sqrt\delta}v_0^2 PU_{\delta}-4\e \int_{|x |<R_0\sqrt\delta}u_0 v_0 PU_{\de,  0 }+\mathcal O\(\delta^4\)\\
%&=\frac 23 \alpha_6^3\omega_6\int_{\frac {R_0}{\sqrt\delta}}^{+\infty}\frac{r^5}{(1+r^2)^6}-2\alpha_6^2\omega_6 u_0(0)\delta^2 \int_{\frac{R_0}{\sqrt\delta}}^{+\infty}\frac{r^5}{(1+r^2)^4}+\frac 23 \delta^6 u_0^3(0)\omega_6\int_0^{\frac{R_0}{\sqrt\delta}}r^5\\
%&-2\alpha_6\omega_6u_0^2(0)\delta^4\int_0^{\frac{R_0}{\sqrt\delta}}\frac{r^5}{(1+r^2)^2}+\mathcal O\(\delta^4\)+\mathcal O\(|\e|\de^3\)+\mathcal O\(\e^2\de^2\)\\
%&=\frac 1 9 \omega_6\alpha_6^3R_0^{-6}\delta^3-\alpha_6^2\omega_6u_0(0) R_0^{-2}\delta^3+\frac 1 9 \omega_6 u_0(0)^3 R_0^6\delta^3-\omega_6\alpha_6 u_0(0)^2R_0^2\delta^3+\mathcal O\(\delta^4|\ln\delta|\)\\
%&=\delta^3\(\frac 19\omega_6\alpha_6^{\frac 32}(u_0(0))^{\frac 32}-\omega_6\alpha_6^{\frac 32}(u_0(0))^{\frac 32}+\frac 19 \omega_6\alpha_6^{\frac 32}(u_0(0))^{\frac 32}-\omega_6\alpha_6^{\frac 32}(u_0(0))^{\frac 32}\)+\mathcal O\(\delta^4|\ln\delta|\)\\
%&=-\frac{16}{9}\omega_6\alpha_6^{\frac 32}(u_0(0))^{\frac 32}\delta^3+\mathcal O\(\delta^4|\ln\delta|\)\end{aligned}$$
Collecting all the previous estimates we get \eqref{cruc1} with
\begin{equation}\label{costanti}
 \mathfrak a_1 =\alpha_6^2\(\int\limits_{\R^6}\frac{1}{(1+|y|^2)^4}dy\)=96\sigma_6 \ \hbox{and}\
 \mathfrak a_2 =\frac{11}{9}\sigma_6\alpha_6^{\frac 32}(\bar u^{m-1}(0))^{\frac 32}
\end{equation}
and 
\[\mathfrak c_0(\e)=I_1+\frac1 6\int\limits_{\R^6}U^3
\]
that ends the proof.
\end{proof}

We are now in position to prove Theorem \ref{main1}.
\begin{proof}[Proof of Theorem \ref{main1}: completed] The claim  follows by Proposition \ref{cruc0} taking into account that
 if $ \mathtt{sgn}(\e)\(1-2\bar v_0(0)\)>0$ the function $\Upsilon$ has always an isolated maximum point
 $$d_0:={2\mathfrak a_1\over 3\mathfrak a_2}\mathtt{sgn}(\e)\(1-2\bar v_0(0)\)$$ which is stable under uniform perturbations. 
\end{proof}

\subsection{Proof of Theorem \ref{teo:N=6}}

%\subsubsection{The validity of the assumptions 	\eqref{30} and \eqref{v00}}\label{ipotesi}
 Here we give the proof of Theorem \ref{teo:N=6}. We start analyzing the assumptions \eqref{30} and \eqref{v00} and the asymptotics of the solution $u_\e$.
	In the case $m=2$  the condition \eqref{30} states the nondegeneracy of a solution to the Brezis-Niremberg problem \eqref{1} with fixed sign, which  was proved in \cite{Sri93}.
	In next proposition we check the validity of the condition \eqref{v00}. 

\begin{proposition}\label{prop-4}
 Assume $\bar u^{m-1}$ is nondegenerate and let
% Let 
$\bar v_0$ solves \eqref{30}. We have that
\begin{align}\label{n2}
		\hbox{for every $m\ge 2$, then  $1-2\bar v_0(0)\neq 0$,} \\
	\label{n1}
	\hbox{if $m=2$, then $1-2\bar v_0(0)>0$}.
\end{align}
	\end{proposition}

 \begin{proof}
Let us start to prove \eqref{n2}. A straightforward computation shows that
$$w_0(r)= \frac r2(\bar u^{m-1})'(r)+ \bar u^{m-1}(r)$$
solves
\begin{equation}
\begin{cases}
-\Delta w_0-(2|\bar u^{m-1}|+\bar \lambda)w_0=\bar \lambda \bar u^{m-1}\ \hbox{in}\ B\\
w_0(1)=\frac12(\bar u^{m-1})'(1)\ \hbox{on}\ \partial B.
 \end{cases}
 \end{equation}
Therefore, the function $z_0=w_0-\bar \lambda v_0$ solves  (in radial coordinates)
\begin{equation}
\begin{cases}
-z_0''(r)-\frac5r z'_0(r)-(2|\bar u^{m-1}(r)|+\bar \lambda)z_0(r)=0\ \hbox{in}\ [0,1],\\
z_0(0)=\bar u^{m-1}(0)\(1-2\bar v_0(0)\)\\
z'_0(0)=0,\\
z_0(1)=\frac12  (\bar u^{m-1})'(1).
 \end{cases}
 \end{equation}
It is immediate to check that $ z_0(0)\not=0$, otherwise by the uniqueness of solutions to the Cauchy problem, we get $z_0\equiv 0$ which is not possible because $z_0(1)\not=0.$ \\
Next we prove \eqref{n1}. This is the case where $m=2$ and $u^{1}>0$ in $(0,1)$,  and the claim $1-2\bar v_0(0)>0$ is equivalent to $z_0(0)>0$.
 By contradiction suppose that $z_0(0)<0$ and, since $z_0(1)<0,$ only  two possibilities occur
\begin{enumerate}
\item  $z_0(r)\le 0$ for any $r\in[0,1]$
\item there exist  $a,b\in(0,1)$ such that $z_0(a)=z_0(b)=0,$ $z_0(r)<0$ for any $r\in[0,a)$ and $z_0(r)>0$ for any $r\in(a,b)$
\end{enumerate}
We will prove that both of them lead to a contradiction.

If $z_0(r)\le 0$ for any $r\in[0,1]$ then by the maximum principle immediately we deduce that $z_0<0$ in $[0,1].$ 
Therefore, it follows that the first eigenvalue of the linear operator $-\Delta -(2u^0+\bar \lambda)I$ is strictly positive (see, for example  \cite{bnv}, p.48) and a contradiction arises since  the Morse index of $u_0$ is 1.
\\
In case 2, set
$$z_1(r)=\left\{\begin{aligned}z_0(r)\ &\hbox{if}\ r\in[0,a]\\ 0\ &\hbox{if}\ r\in[a,1]
\end{aligned}\right.\ 
 \hbox{and}\ z_2(r)=\left\{\begin{aligned}0\ &\hbox{if}\ r\in[0,a]\cup[b,1]\\ z_0(r)\ &\hbox{if}\ r\in[a,b]\end{aligned}\right.$$
It is clear tha $z_1,z_2\in H^1_{0,{\mathtt rad}}( B)$ and they are linearly independent. If $E={\mathtt span}\{z_1,z_2\}$ then $\mathtt dim E=2$ and 
the quadratic form
$$\mathcal Q(z):=\int\limits_ B\( |\nabla z|^2-(2|\bar u^{m-1}|+\bar \lambda)z^2\)dx$$
vanishes over $E.$ By the variational characterization of the Morse index, we deduce that the second eigenvalue of $-\Delta -(2|\bar u^{m-1}|+\bar \lambda)I$ is less or equal than zero.  Since $\bar u^{m-1}$ is non-degenerate and it has Morse index $1$ we get a contradiction.
\end{proof}

%\subsection{The asymptotics of the solution $u_\e$}
The behavior of $||u_\e||_\infty$ is described in the following lemma.
\begin{lemma}\label{n20}
With the same assumption as in Theorem \ref{main1} we have that
\begin{equation}\label{n12}
||u_\e||_\infty= -u_\e(0)\sim121\frac{\bar \l^3}{8\big(1-2v_0(0)\big)^2}\frac1{\e^2}.
\end{equation}
\end{lemma}
\begin{proof}
We have to compute $||u_\e||_\infty=||W_\de+\phi_\e||_\infty$ {where $W_\de$ is as defined in \eqref{sol}, $\de$ as in \eqref{sceltapar} and $d$ as in \eqref{scelta-d}.}
First let us write the equation satisfied by $\phi_\e$,
\begin{equation}
\begin{split}
&-\Delta(\bar u^{m-1}+ \e \bar v_0-PU_{\delta}+\phi_\e)=|(\bar u^{m-1}+\e \bar  v_0-PU_{\delta}+\phi_\e|(\bar u^{m-1}+\e \bar v_0-PU_{\delta}+\phi_\e)+\\
&(\bar \lambda+\e)(\bar u^{m-1}+\e \bar v_0-PU_{\delta}+\phi_\e)\Rightarrow\\
&-\Delta\phi_\e-(\bar \lambda+\e)\phi_\e=|(\bar u^{m-1}+\e \bar v_0-PU_{\delta}+\phi_\e|(\bar u^{m-1}+\e \bar v_0-PU_{\delta}+\phi_\e)+U_\de^2\\
&-|\bar u^{m-1}|\bar u^{m-1}-(\bar \lambda+\e)PU_{\delta}+\e\bar u^{m-1}+\e^2 v_0.
\end{split}
\end{equation}
Next we scale the equation setting $\tilde\phi_\e(x)=\de^2\phi_\e(\de x)$, $\tilde\phi_\e:B\left(0,\frac1\de\right)\to\R$. It verifies the equation
\begin{equation}
\begin{split}
&-\Delta\tilde\phi_\e=-\de^4\Delta\phi_\e(\de x)=
(\bar \lambda+\e)\de^2\tilde\phi_\e+\left|O(\de^2)-\underbrace{\de^2PU_{\delta}(\de x)}_{=U(x)+o(1)}+\tilde\phi_\e\right|\left(O(\de^2)-\underbrace{\de^2PU_{\delta}(\de x)}_{=U(x)+o(1)}+\tilde\phi_\e\right)\\
&+U^2(x)-(\bar \lambda+\e)\de^4PU_{\delta}(\de x)+O(\de^4).
\end{split}
\end{equation}
Using \eqref{a1} with $a=-U$ and $b=\tilde\phi_\e+o(1)$ we get
\begin{equation}
\left|-U(x)+\tilde\phi_\e+o(1)\right|\left(-U(x)+\tilde\phi_\e+o(1)\right)+
U^2(x)=2U(x) \left(\tilde\phi_\e+o(1)\right)
+O\left(|\tilde\phi_\e+o(1)|^2\right)
\end{equation}
and then
\begin{equation}\label{stim}
-\Delta\tilde\phi_\e=
(\bar \lambda+\e)\delta^2\tilde\phi_\e+2U(x)
\left(\tilde\phi_\e+o(1)\right)
+O\left(|\tilde\phi_\e+o(1)|^2\right)-(\bar \lambda+\e)\delta^4PU_{\delta}(\delta x)+O(\delta^4).
\end{equation}
We claim that the RHS of \eqref{stim} goes to $0$ in $L^2(B(0,R))$ for any $R>0$.
Indeed, using the the Poincaré inequality,
$$\int_{B_R}|\tilde\phi_\e|^2dy=\delta^4\int_{B_R}|\phi_\e(\delta y)|^2dy=
\frac1{\delta^2}\int_{\delta B_R}\phi^2dx\le(\hbox{by \eqref{stimaphi}})\le
\frac{\delta^4|\ln|\delta||^{\frac 43}} {\delta^2}=o(1)$$
and as before using that $0\le PU_\de\leq U_\de$,
$$\delta^4\left(\int_{B_R}|PU_{\delta}(\delta x)|^2\right)^{\frac 12}\le\delta^2\left(\int_{B_R}U^2(x)\right)^{\frac 12}=o(1).$$
Then by the standard regularity theory we get that $\tilde\phi_\e\to0$ uniformly on $B(0,R)$ and then
\begin{equation}
\tilde\phi_\e(0)\to0\Rightarrow\phi_\e(0)=o\left(\frac1{\delta^2}\right).
\end{equation}
Of course this estimate is far to be sharp but it is enough for our aims. Indeed we get from \eqref{sol}
\begin{equation}\label{n21}
\begin{split}
&\delta^2||u_\e||_\infty=-\delta^2u_\e(0)=-\delta^2\bar u^{m-1}(0)-\delta^2\e \bar v_0(0)+\delta^2PU_{\delta}(0)-\delta^2\phi_\e(0)=\\
&\left(\hbox{recalling that }PU_{\delta}(x)=\frac{\alpha_6\delta^2}{(\delta^2+|x|^2)^2}+O(\delta^2)\right)=\\
&\alpha_6+o(1)-\delta^2\phi_\e(0)=\alpha_6+o(1)\Rightarrow\\
&||u_\e||_\infty\sim\frac{\alpha_6}{d^2\e^2}=\frac{121}8\frac{\bar\l^3}{\big(1-2\bar v_0(0)\big)^2}\frac1{\e^2}
\end{split}
\end{equation}
which ends the proof.
\end{proof}
\vskip0.2cm
We have now all the ingredients to conclude the  proof of Theorem \ref{teo:N=6}.

\begin{proof}[\bf Proof of Theorem \ref{teo:N=6}]
If $u_\l$ is any radial solution to \eqref{1} which is positive in the origin and has $m>1$ nodal zones, the assumption $\|u_\l\|_\infty=u_\l(0)\xrightarrow{\l\to \bar \l} \infty$ is equivalent to $r_\l\to 0$, thanks to Lemma \ref{lem-preliminare-N=4,5,6}. In  Lemma \ref{lem-6-1} we have showed that $\bar \l$ is characterized by \eqref{i10}, proving  \eqref{norma-A-N=6} and \eqref{u-A-N=6}.
\\
Next we check that $u_\l$ coincides with $-u_\e$ defined in \eqref{n11}, for $\e=\l-\bar\l$.  Observe that Proposition \ref{prop-4} states that $2v_0(0)\neq -1$ ($v_0$ given by \eqref{v0i}), hence the only assumption that $u^{m-1}$ is nondegenerate is needed to apply  Theorem \ref{main1}. Recall that the nondegeneracy of $u^{m-1}$ for $m=2$ was proved in \cite{Sri93}.
Let us denote by 
$$K=\left\{\frac{-u_\e(0)}{ \bar\l+\e}\hbox{ with }|\e| \in(0,\e_0) \hbox{ and  $\e_0$, $\e$ as in Theorem \ref{main1}}\right\}=(K_0,+\infty)$$
for some $K_0>0$ by \eqref{n12}. In the present setting $\frac{u_\l(0)}{\l}\subset K$ for $\l$ close to $\bar\l$, hence Proposition \ref{un} yields that  $\l=\bar\l+\e$ and $u_\l=-u_\e$.
\\
So \eqref{norma-B-N=6} follows by Lemma \ref{n20} and \eqref{n21} , recalling that $v_0=-\bar v_0$. Inserting \eqref{norma-B-N=6} into \eqref{norma} gives
$$r_\l^4=\frac{1152}{\bar\l||u_\l||_\infty}\sim\frac{9216\big(1+2v_0(0)\big)^2} {121\bar \l^4} (\l-\bar\l)^2, $$
which brings to \eqref{r-N=6}. Eventually Theorem \ref{main1} also states that   $\e=\l-\bar \l$ has the same sign of $1-2\bar v_0(0)=1+2v_0(0)$. In particular when $m=2$ Proposition \ref{prop-4} guarantees that $1-2\bar v_0(0)>0$, so that $\l\to \bar \l$ from above. Moreover in this particular case the nondegeneracy of $u^{m-1}$ is known by \cite{Sri93}.
\end{proof}
\bibliographystyle{abbrv}
\bibliography{AGGPV2}
 
\end{document}